\documentclass[12pt,reqno]{amsart}
\usepackage[T1]{fontenc}
\usepackage[utf8]{inputenc}
\usepackage{lmodern}
\usepackage[margin=1in]{geometry}
\usepackage{amsmath,amssymb,mathtools,microtype}
\usepackage{booktabs,array,longtable,enumitem,needspace,etoolbox}
\usepackage{tikz,tikz-cd,listings}
\usetikzlibrary{arrows.meta,positioning}
\usepackage[bookmarksnumbered,hidelinks]{hyperref}
\hypersetup{pdftitle={Recognizing oriented knots from finite quotients with one marked element},
 pdfauthor={Indranil Biswas and Shrinit Singh}}
\DeclareMathOperator{\Inn}{Inn}

\DeclareMathOperator{\GL}{GL}
\DeclareMathOperator{\PSL}{PSL}
\DeclareMathOperator{\diag}{diag}
\DeclareMathOperator{\lk}{lk}

\DeclareMathOperator{\im}{im}
\DeclareMathOperator{\id}{id}

\DeclareMathOperator{\coeff}{coeff}

\newcommand{\R}{\mathbb R}

\newcommand{\Z}{\mathbb Z}
\newcommand{\C}{\mathbb C}
\newcommand{\Zh}{\widehat{\mathbb Z}}
\newcommand{\cl}[1]{\overline{#1}}
\newcommand{\Pk}[1]{\mathcal P(#1)}
\newcommand{\Qk}[1]{\mathcal Q(#1)}
\newcommand{\Rk}[1]{\mathcal R(#1)}
\newcommand{\Wk}[1]{\mathcal W(#1)}
\newcommand{\Qwk}[1]{\mathcal Q_w(#1)}
\newcommand{\into}{\hookrightarrow}
\newcommand{\onto}{\twoheadrightarrow}
\numberwithin{equation}{section}
\setlist[enumerate,1]{label=\textup{(\roman*)},leftmargin=2.3em,itemsep=3pt}
\setlist[itemize,1]{leftmargin=2em,itemsep=2pt}
\newtheorem{thm}{Theorem}[section]
\newtheorem{lem}[thm]{Lemma}
\newtheorem{prop}[thm]{Proposition}
\newtheorem{cor}[thm]{Corollary}
\newtheorem{maintheorem}{Theorem}

\theoremstyle{definition}

\theoremstyle{remark}

\BeforeBeginEnvironment{thm}{\Needspace{5\baselineskip}}
\BeforeBeginEnvironment{lem}{\Needspace{5\baselineskip}}
\BeforeBeginEnvironment{prop}{\Needspace{5\baselineskip}}
\BeforeBeginEnvironment{cor}{\Needspace{5\baselineskip}}
\BeforeBeginEnvironment{maintheorem}{\Needspace{6\baselineskip}}
\BeforeBeginEnvironment{defn}{\Needspace{4\baselineskip}}
\BeforeBeginEnvironment{ex}{\Needspace{4\baselineskip}}
\BeforeBeginEnvironment{rem}{\Needspace{4\baselineskip}}
\title[One marked element and oriented knots]
 {Recognizing oriented knots from finite quotients
 with one marked element}
 
\author{Indranil Biswas}
\address{Department of Mathematics, Shiv Nadar Institution of Eminence Deemed to be University, Dadri, Greater Noida 201314, India}
\email{indranil.biswas@snu.edu.in, indranil29@gmail.com}

\author{Shrinit Singh}
\address{International Centre for Theoretical Sciences--Tata Institute
of Fundamental Research, Bangalore 560089, India}
\email{shrinit.singh@icts.res.in, shrinitsingh@gmail.com}
\date{}
\begin{document}

\begin{abstract}
For an oriented tame knot $K$, let $G_K$ be its group, with positive meridian $\mu_K$ and preferred longitude $\lambda_K$. We prove that the profinite completion $\widehat{G_K}$, marked by the conjugacy class of $\mu_K\lambda_K$, determines $K$ up to oriented equivalence. In fact, every isomorphism preserving this mark preserves the meridian and longitude with the same conjugator. The proof combines Xu's genuineness theorem with profinite JSJ theory and simultaneous boundary realization, and applies to satellite and composite knots. Consequently, Eisermann's colouring polynomials over all finite groups determine oriented knots.
\end{abstract}

\maketitle
\raggedbottom

\section{Introduction}\label{sec:introduction}

Let $K$ be an oriented tame knot in the oriented $3$-sphere. Choose a
closed tubular neighbourhood $\nu K$, a solid torus with core $K$, and put
\[
 E_K\ =\ S^3\setminus\operatorname{int}(\nu K),\qquad G_K\ =\ \pi_1(E_K).
\]
The positive meridian $\mu_K$ bounds a normal disc meeting $K$ positively
once. The preferred longitude $\lambda_K$ follows $K$ positively once and
has linking number zero with it. Transporting both boundary curves to a
basepoint along one path gives commuting elements of $G_K$; changing
that path conjugates both by the same element. Our conventions are
specified in Section~\ref{sec:topology}.

The discrete peripheral system $(G_K,\,\mu_K,\,\lambda_K)$ determines the
oriented knot, by Waldhausen's realization theorem and meridional filling
\cite[Corollary~6.5]{Waldhausen}; see also \cite[Section~2]{Eisermann}.
The abstract group does not suffice: the square and granny knots have
isomorphic groups. We ask how much of this peripheral information must
be retained when the group is observed through its finite quotients.

Write
\[
 \widehat G\,\ =\,\ \varprojlim_{N\trianglelefteq_fG}G/N,
\]
where $N\,\trianglelefteq_f\, G$ means that $N$ is a normal subgroup of finite index.
Every homomorphism from $G$ to a finite group extends uniquely to a
continuous homomorphism from $\widehat G$. Knot groups are residually
finite, so they embed in their completions; we use the same symbols for
peripheral elements and their images. Define
\[
 \Pk K\ :=\ (\widehat{G_K},\,\mu_K,\,\lambda_K),\quad
 w_K\ :=\ \mu_K\lambda_K,\quad \Wk K\ :=\ (\widehat{G_K},\, w_K).
\]
An isomorphism of ordered markings is a continuous isomorphism $\Phi$
for which
\begin{equation}\label{eq:exact-marking}
 \Phi(\mu_K)\ =\ g\mu_Lg^{-1},\quad
 \Phi(\lambda_K)\ =\ g\lambda_Lg^{-1}
\end{equation}
for one $g\,\in\,\widehat{G_L}$. An isomorphism of one-element markings
requires $\Phi(w_K)\,=\,gw_Lg^{-1}$. 
Throughout, $K\,\simeq\, L$ means an orientation preserving homeomorphism of
$S^3$ taking $K$ to $L$ and preserving their knot directions, equivalently
they are oriented ambient isotopic. Our first result is a recognition theorem with both orientations retained.

\begin{maintheorem}\label{thm:main}
For oriented tame knots $K,\,L\,\subset\, S^3$, the following three statements are
equivalent:
\begin{enumerate}
\item $K\ \simeq\ L$;

\item $\Wk K\ \cong\ \Wk L$;

\item $\Pk K\ \cong\ \Pk L$.
\end{enumerate}
\end{maintheorem}

The conclusion includes the unknot and arbitrary satellite and composite
knots. Its geometric input is the following theorem. Here
$\PSL_2(\C)$ is the orientation-preserving isometry group of hyperbolic
$3$-space; a lattice is a discrete subgroup with finite-volume quotient,
nonuniform means that the quotient is noncompact, and torsion-free means
that no nonidentity element has finite order.

\begin{thm}[Xu's genuineness theorem
{\cite[Definition~6.1 and Proposition~16.3]{Xu26}}]
\label{thm:xu-genuineness}
Let $\Gamma,\,\Delta$ be torsion-free nonuniform lattices in $\PSL_2(\C)$.
Every continuous isomorphism $f\,:\,\widehat\Gamma\,\longrightarrow\,\widehat\Delta$
satisfies the following:
\[
 f\ =\ \Inn(h)\circ\widehat\alpha,
 \quad h\ \in \widehat\Delta,
 \quad\alpha\,:\, \Gamma\,\xrightarrow{\,\,\cong\,\,}\, \Delta.
\]
The map $\widehat\alpha$ is induced by the discrete group isomorphism
$\alpha$, and $\Inn(h)(z)\,=\,hzh^{-1}$.
\end{thm}

Theorem~\ref{thm:xu-genuineness} identifies the particular profinite isomorphism, up to an inner automorphism. Together with Mostow--Prasad rigidity, it gives a homeomorphism of each hyperbolic piece realizing all of its prescribed cusp maps. For satellite and composite exteriors, the maps on these pieces
must agree with those on neighbouring Seifert pieces. We group each maximal connected collection of Seifert pieces into a graph-manifold block. The integral gluing matrices make the prescribed map on the block side of each hyperbolic interface integral. Xu's simultaneous graph-manifold boundary theorem gives two homeomorphisms of the block and a single profinite unit relating their boundary maps to prescribed maps. Integrality forces this unit to be $1$ or $-1$, so one of the two homeomorphisms realizes every prescribed boundary map of the block. The resulting piece maps must agree up to isotopy on the common tori. Adjusting them in boundary collars allows them to glue while retaining the specified exterior boundary action.  

\begin{maintheorem}\label{thm:one-recovery}
Let $\Phi\,:\,\widehat{G_K}\,\longrightarrow\, \widehat{G_L}$ be a continuous isomorphism
for which $\Phi(\mu_K\lambda_K)\,=\,g\mu_L\lambda_Lg^{-1}$.
Then this $g\,\in\, \widehat{G_L}$ satisfies the following:
\[
 \Phi(\mu_K)\ =\ g\mu_Lg^{-1},\quad \Phi(\lambda_K)\ =\ g\lambda_Lg^{-1}.
\]
\end{maintheorem}

For a nontrivial knot, the centralizer of $w_K$ is precisely the closed
peripheral subgroup. An isomorphism preserving $w_K$ therefore preserves this subgroup. Since $w_K$ has image $1$ under the normalized abelianization, the geometry of the root JSJ piece then forces the induced boundary matrix to be integral. A Dehn-filling argument identifies the meridian and longitude axes up to signs; the equality $\Phi(w_K) = w_L$ fixes both
signs after one conjugation. Theorem~\ref{thm:one-recovery} uses the earlier peripheral results and is independent of
Theorem~\ref{thm:xu-genuineness}. The same argument works for the exact mark $\mu_K\lambda_K^n$ with any fixed integer $n\,\ne\,0$.

\subsection{Finite representations}

For a finite group $Q$ and $x\,\in\, Q$, Eisermann's colouring polynomial is
\begin{equation}\label{eq:colouring-poly}
 P_Q^x(K)\,\ =\,\ \sum_{\substack{\rho:G_K\to Q\\\rho(\mu_K)=x}}
\rho(\lambda_K)\in\Z[Q].
\end{equation}
The group ring $\Z[Q]$ consists of formal integer sums of elements of
$Q$. The coefficient of $y$ counts homomorphisms with meridian image
$x$ and longitude image $y$. For the one-element mark put
\[
 N_Q(z;K)\ =\ \#\{\rho\,:\,G_K\,\to\, Q\,\,\big\vert\,\, \rho(w_K)\,=\,z\}.
\]
Also, let $\Qk K$ and $\Qwk K$ be the sets of isomorphism classes
\[
 \begin{split}
 \Qk K\ &=\ \{[(Q,q(\mu_K),\,q(\lambda_K))]\,\, \big\vert\,\, q\,:\,G_K\,\onto\, Q,\,\ Q\,\text{ finite}\},\\
 \Qwk K\ &=\ \{[(Q,\,q(w_K))]\,\,\big\vert\,\, q\,:\,G_K\,\onto\, Q,\,\ Q\,\text{ finite}\}.
 \end{split}
\]
The arrow denotes a surjection, and an isomorphism of marked finite groups preserves every displayed element. We record which isomorphism classes occur, without recording their multiplicities or the maps between quotients.

\begin{cor}\label{cor:colouring-main}
For oriented tame knots $K,\,L$, the following are equivalent:
\begin{enumerate}
\item $K\,\simeq\, L$;
\item $\Qwk K\,=\,\Qwk L$;
\item $N_Q(z;K)\,=\,N_Q(z;L)$ for every finite $Q$ and $z\,\in\, Q$;
\item $\Qk K\,=\,\Qk L$;
\item $P_Q^x(K)\,=\,P_Q^x(L)$ for every finite $Q$ and $x\,\in\, Q$.
\end{enumerate}
\end{cor}

This answers \cite[Question~1.5]{Eisermann}.

\subsection{Prime branches and reduced marking}

Schubert's theorem gives a finite prime decomposition, unique up to
permutation and oriented equivalence \cite{Schubert}. Write
$\Zh\,=\,\varprojlim_n\Z/n\Z$ for the profinite integers and $\Zh^\times$ for
its units. For a profinite group element $z$, the power $z^u$ is specified
in each finite quotient by reducing $u$ modulo the order of the image
of $z$.

\begin{maintheorem}\label{thm:constituents}
Let $K\,=\,K_1\#\cdots\#K_r$ and $L\,=\,L_1\#\cdots\#L_s$ be decompositions
into nontrivial prime knots, with $r\,\ge\,2$. Suppose a continuous
isomorphism $\Phi\,:\,\widehat{G_K}\,\longrightarrow\, \widehat{G_L}$ satisfies, up to
simultaneous conjugacy,
\[
 \Phi(\mu_K)\ =\ \mu_L^\alpha,\quad
 \Phi(\lambda_K)\ =\ \lambda_L^\beta,
 \quad\alpha,\,\beta\ \in\ \Zh^\times.
\]
Then $r\,=\,s$, and there exist a permutation $\sigma$ of $\{1,\,\cdots,\,r\}$
and continuous isomorphisms
$\Phi_i\,:\,\widehat{G_{K_i}}\,\longrightarrow\,\widehat{G_{L_{\sigma(i)}}}$ such that,
up to simultaneous conjugacy for each $i$,
\[
 \Phi_i(\mu_{K_i})\ =\ \mu_{L_{\sigma(i)}}^\alpha,
 \qquad\Phi_i(\lambda_{K_i})\ =\ \lambda_{L_{\sigma(i)}}^\beta.
\]
\end{maintheorem}

The exponents $\alpha$ and $\beta$ are the same for every prime branch, including branches
containing several satellite pieces. In the composing region, global abelianization removes meridional shears, and abelianizing its planar
base forces all longitude exponents to equal the outer exponent $\beta$.

Let $K^\dagger$ be the mirror knot with knot direction also reversed, and put
\[
 \Rk K\ =\ (\widehat{G_K},\, \mu_K,\,\cl{\langle\lambda_K\rangle}).
\]
An isomorphism of reduced markings preserves the exact meridian and the
closed longitude subgroup by one common conjugation.

\begin{maintheorem}\label{thm:reduced}
For nontrivial knots, a continuous isomorphism
$\Phi\,:\,\widehat{G_K}\,\longrightarrow\,\widehat{G_L}$ satisfying
$\Phi(\mu_K)\,=\,\mu_L$ and $\Phi(\lambda_K)\,=\,\lambda_L^\beta$ up to
simultaneous conjugacy, with $\beta\,\in\,\Zh^\times$, has $\beta\,\in\,\{1,\,-1\}$.

For arbitrary oriented tame knots,
\[
 \Rk K\ \cong\ \Rk L
 \]
 if and only if either $K\,\simeq\, L$ or $K\,\simeq\, L^\dagger$.
\end{maintheorem}

The restriction $\beta\,\in\,\{1,\,-1\}$ follows from the peripheral results and Theorem~\ref{thm:constituents}; knot recognition then follows from Theorem \ref{thm:main}. For graph knots, whose JSJ pieces are all
Seifert, recognition uses graph-manifold boundary realization. Replacing the knot by its reversed mirror changes all prime summands together. 

\subsection{Relation to earlier works}

Profinite JSJ theory is due to Wilton--Zalesskii and Wilkes
\cite{WZGraph,WZGeometries,WZ,WilkesJSJ}; Seifert fibre and scale analysis
is developed in \cite{WilkesSeifert,WilkesGraphI}. We use Xu's precise
congruent-graph and boundary-realization formulation \cite{Xu25}.
Boileau--Friedl recognize torus-knot exteriors from unmarked completions
among knot exteriors, allowing reflection
\cite[Theorem~1.5]{BoileauFriedl}. Wilkes proves unmarked group rigidity
for graph-knot exteriors, and exterior recognition for prime graph knots
\cite[Theorem~4.1]{WilkesKnots}. We realize the boundary action prescribed by the profinite isomorphism, which allows the exterior homeomorphism to preserve both orientations after filling. We also show that the single element $\mu_K\lambda_K$ recovers the ordered pair and that the same peripheral exponents pass to every prime summand.

After fixing conventions, we show that finite quotients with marked elements determine the marked profinite completion. We then establish the rooted JSJ restrictions and prove that the same peripheral exponents pass to every prime summand. The following sections construct exterior homeomorphisms realizing the prescribed boundary maps and recover the ordered meridian–longitude pair from a single marked element. We then deduce the results on reduced markings and colouring polynomials. The final section gives explicit finite witnesses distinguishing the square and granny knots, with a verification program in the appendix.
\section*{Acknowledgement}
The second author thanks Professor Martin Bridson for a helpful discussion at TIFR Mumbai in November 2025 concerning the problem of distinguishing knots using finite quandles. His suggestion to approach the problem from a group-theoretic perspective and his interest in the question were particularly encouraging. The second author acknowledges support from the Department of Atomic Energy, Government of India, under project no. RTI4019.
\section*{AI Declaration}
The authors acknowledge the use of OpenAI’s ChatGPT for language editing, for literature search and for refining some of the arguments. None of the ideas were given by AI. The mathematical ideas and results are those of the authors, who take full responsibility for the correctness and content of the paper.

\section{Conventions and profinite preliminaries}\label{sec:topology}
\label{sec:profinite}

\subsection{Exteriors and orientations}\label{subsec:paths}

A knot is the image of an embedding $S^1\,\into\, S^3$. It is called tame if an
ambient isotopy takes it to a polygonal knot, namely, a finite union of straight
line segments in $\R^3\,\subset\, S^3$. Recall that an ambient isotopy is a family of
homeomorphisms $H_t\,:\,S^3\,\longrightarrow\, S^3$, continuous in $t$, with $H_0\,=\,\id$.
We work in the usual smooth or piecewise-linear category. A tubular
neighbourhood is a homeomorphism of pairs
$(\nu K,\,K)\,\cong\,(S^1\times D^2,\,S^1\times\{0\})$.
The exterior $E_K$ is compact with the boundary $T_K\,=\,\partial\nu K$;
the complement $S^3\setminus K$ retracts onto it and has the same
fundamental group. The unknot $U$ bounds an embedded disc and has
$G_U\,\cong\,\Z$ and $\lambda_U\,=\,1$.

The {\it linking number} $\lk(C,D)$ of two disjoint oriented curves $C,\,D$
is the signed intersection number of $D$ with an oriented spanning surface
$F$ for $C$. An intersection is positive when a positive tangent to $D$
followed by an oriented tangent basis of $F$ gives the ambient
orientation. Since $H_2(S^3;\,\Z)\,=\,0$, the linking number $\lk(C,D)$
is independent of the choice of $F$.
A Seifert surface is such an oriented embedded spanning surface  for $K$.
The meridian $\mu_K$ and preferred longitude $\lambda_K$ are characterized by
$\lk(\mu_K,\,K)\,=\,1$ and $\lk(\lambda_K,\,K)\,=\,0$, with the longitude following
$K$ positively once. They form an ordered basis of $H_1(T_K;\,\Z)$.

Choose a basepoint $p\,\in\, E_K$, a point $q\,\in\, T_K$, and a path
$\gamma\,:\,p\,\longrightarrow\, q$. The based loop corresponding to a boundary loop $c$
is $\gamma*c*\gamma^{-1}$. For another path $\gamma' : p \longrightarrow q$, the replacement of $\gamma$ with $\gamma'$ conjugates $\mu_K$ and $\lambda_K$ by the single element
$[\gamma'*\gamma^{-1}]$. We use one path per boundary component throughout. Thus,
\[
 P_K\ =\ \im\bigl(\pi_1(T_K)\,\to\, G_K\bigr)\ =\ \langle\mu_K,\,\lambda_K\rangle
\]
is uniquely defined up to conjugacy, and its generators up to simultaneous
conjugacy.

A compact $3$-manifold $M$ is irreducible if every embedded sphere bounds a
ball. A two-sided surface $S$ in $M$ is incompressible if it has no compressing
disc, meaning a disc meeting it only in an essential boundary curve. The
Loop Theorem gives that $S$ is incompressible if and only if 
the homomorphism $\pi_1(S)\,\longrightarrow\, \pi_1(M)$ is
injective \cite{Papakyriakopoulos,StallingsLoop}. A surface is
boundary-parallel if it cobounds a product region with the boundary;
an essential torus or an annulus is incompressible and not boundary-parallel.
Knot exteriors are irreducible, and the boundary of a nontrivial knot
exterior is incompressible \cite[Chapters~1--3]{AFW}. Consequently
$P_K\,\cong\,\Z^2$ for $K\,\ne\, U$.

For a knot, the linking number gives its normalized abelianization
\begin{equation}\label{eq:chi-discrete}
 \chi_K\ :\ G_K\ \longrightarrow\ \Z,
 \quad\chi_K(\mu_K)\ =\ 1,
 \quad\chi_K(\lambda_K)\ =\ 0.
\end{equation}
The meridian normally generates $G_K$, and the longitude belongs to its
commutator subgroup. These facts follow, for example, from a Wirtinger
presentation, whose generators are meridians and whose crossing
relations are conjugations.

\subsection{Slopes and filling}\label{subsec:classical}

A slope is an unoriented essential simple closed curve on a torus up to
isotopy. In an integral basis it is a primitive vector $(a,\,b)\,\in\,\Z^2$,
with $(a,\,b)$ and $(-a,\,-b)$ being equivalent. We use column vectors:
$\mu\,\longmapsto\,\mu'^a\lambda'^c$ and
$\lambda\,\longmapsto\,\mu'^b\lambda'^d$ give the matrix
$\left(\begin{smallmatrix}a&b\\c&d\end{smallmatrix}\right)$.
Every element of $\GL_2(\Z)$ is realized by a torus homeomorphism; two torus
homeomorphisms with the same matrix are isotopic. On source and target
exteriors the meridian--longitude basis has the same sign relative to
boundary orientation, so elements of $\GL_2(\Z)$ of determinant $+1$
preserve this orientation.

Dehn filling along a slope attaches a solid torus whose meridional disc
has that boundary slope. The meridional filling recovers $(S^3,\,K)$.
A boundary homeomorphism extends over the attached solid torus when it
preserves the meridian slope. The longitude coefficient controls the
orientation of its core. In particular, an exterior homeomorphism with
boundary matrix $I$, the $2\, \times \,2$ identity matrix in the ordered meridian-longitude bases, extends to an oriented equivalence of the knots.

The Gordon--Luecke theorem, \cite{GordonLuecke}, recognizes the underlying
unoriented knot from its oriented exterior. Allowing arbitrary exterior
homeomorphisms also allows for an ambient reflection. For discrete group
recognition, a minimal-genus Seifert surface makes a nontrivial exterior
sufficiently large, meaning that it contains a properly embedded two-sided
incompressible surface other than a sphere or disc. Waldhausen's
realization theorem then applies to a
peripheral-preserving isomorphism. The exact ordered marks give the matrix $I$,
so filling preserves both orientations
\cite[Corollary~6.5]{Waldhausen}.
Reversing the knot direction inverts both peripheral elements;
mirroring while transporting its direction inverts the meridian only.
Thus, the natural reversed-mirror map has
\begin{equation}\label{eq:dagger-marks}
 \mu_K\ \longmapsto\ \mu_{K^\dagger},\qquad
 \lambda_K\ \longmapsto\ \lambda_{K^\dagger}^{-1}.
\end{equation}

The connected sum joins two knots after removing small oriented arcs.
A nontrivial knot is prime if it is not a connected sum of two
nontrivial knots. Schubert's uniqueness theorem uses oriented summands
\cite{Schubert}. A connected-sum sphere meets the knot twice and becomes
an annulus in its exterior; even a composite exterior is irreducible.

\subsection{Completion, closure, and peripheral topology}

The inverse limit defining $\widehat G$ is a closed subgroup of the
product of its finite quotients, with the product topology. It is a
compact Hausdorff totally disconnected group. A profinite group is a
group of this form. Its open normal subgroups form a basis of
neighbourhoods of the identity element, and it is the inverse limit of the associated
finite quotients. A discrete group is residually finite if every
nonidentity element survives in some finite quotient. Knot groups, and
more generally compact $3$-manifold groups, are residually finite
\cite{HempelRF,AFW}.

A homomorphism $f\,:\,G\,\longrightarrow\, H$ induces a continuous map
$\widehat f\,:\,\widehat G\, \longrightarrow\,\widehat H$. More generally, any homomorphism
from $G$ to a profinite group extends uniquely to its completion: extend
in each finite continuous quotient and take the inverse limit.
Completion preserves direct products, since finite images of the two
factors commute. A profinite group is topologically finitely generated
if finitely many elements generate a dense subgroup. Such a group has
only finitely many maps to each finite group and finitely many open
normal subgroups of any bounded index. Knot completions have this
property because knot groups are finitely presented.

For $H\,\le\, G$, write $\overline H$ for its closure in $\widehat G$.
A subgroup is separable if it is closed in the profinite topology of
$G$.
The induced topology on $H$ is full if each finite-index subgroup of
$H$ contains $H\cap W$ for some finite-index subgroup $W\,\le\, G$.
Fullness implies identification $\widehat H$ with $\overline H$; separability alone
does not assert this. For incompressible toral boundary, we use the following statement.

\begin{thm}[{\cite[Lemma~6.1]{Xu25}}]\label{thm:peripheral-completion}
For a compact orientable $3$-manifold $M$ with incompressible toral
boundary, and a boundary torus $T$, inclusion identifies
$\widehat{\pi_1(T)}$ with $\overline{\pi_1(T)}\,\le\,\widehat{\pi_1(M)}$.
In particular, this closed subgroup is isomorphic to $\Zh^2$.
\end{thm}

The proof of Theorem \ref{thm:peripheral-completion} uses abelian subgroup
separability \cite{Hamilton}. In particular,
\begin{equation}\label{eq:completed-periphery}
 \overline{P_K}\cong\Zh^2,\qquad
 \overline{\langle\mu_K\rangle}\cong\Zh,\qquad
 \overline{\langle\lambda_K\rangle}\cong\Zh\quad(K\ne U).
\end{equation}
The vector $(a,\,b)$ denotes $\mu_K^a\lambda_K^b$.

\subsection{Powers, normal closures, and abelianization}\label{subsec:powers}

For the ring of profinite integers,
$\Zh\,\cong\,\prod_p\Z_p$, where $\Z_p\,=\,\varprojlim_k\Z/p^k\Z$.
A unit is invertible in every factor. An ordinary integer is a unit
precisely when it is $1$ or $-1$. Although $\Zh$ has zero divisors,
multiplication by a nonzero ordinary integer is injective.
For $z$ in a profinite group, the finite-quotient powers defining $z^u$
give the continuous map $\Zh\,\longrightarrow\,\overline{\langle z\rangle}$.
A closed group generated topologically by one element is procyclic.
When it is a copy of $\Zh$, its generators are exactly the unit powers
of any chosen generator \cite{RibesZalesskii}.

Write $\langle\!\langle S\rangle\!\rangle_G$ for the normal subgroup
generated by all conjugates of $S$, and denote by $\overline{\langle\!\langle S\rangle\!\rangle}_G$
its closed normal closure in a profinite group. The following holds without assuming that $G/\langle\!\langle S\rangle\!\rangle_G$ is residually finite.

\begin{lem}\label{lem:quotient-completion}
For any subset $S\,\subseteq\, G$,
\[
 \widehat{\,G/\langle\!\langle S\rangle\!\rangle_G\,}\ \cong\ \widehat G/
 \overline{\langle\!\langle S\rangle\!\rangle}_{\widehat G}.
\]
\end{lem}

\begin{proof}
Homomorphisms from either side to a finite group are exactly homomorphisms from $G$
with $S$ lying in the kernel. Continuity makes the kernel closed, so the completed maps
annihilate the closed normal closure. Their compatible inverse limits agree.
\end{proof}

The meridian therefore closed normally generates $\widehat{G_K}$.
The profinite abelianization is the quotient by the closure of the
commutator subgroup and is the completion of the discrete
abelianization. Thus the extension of $\chi_K$ is
\begin{equation}\label{eq:chi-completed}
 \widehat\chi_K\,:\,\widehat{G_K}\,\longrightarrow\,\Zh,\quad
 \widehat\chi_K(\mu_K)\,=\,1,\quad\widehat\chi_K(\lambda_K)\,=\,0.
\end{equation}
Every continuous automorphism of the additive group $\Zh$ is
multiplication by a unit. Hence an isomorphism with meridian image
conjugate to $\mu_L^\alpha$ satisfies the condition
\begin{equation}\label{eq:chi-compatible}
 \widehat\chi_L\Phi\,\ =\,\ \alpha\widehat\chi_K.
\end{equation}

\subsection{Integral boundary maps}

A continuous homomorphism between finite powers of $\Zh$ is
$\Zh$-linear, by continuity and density of integer scalars. A completed
boundary isomorphism has matrix $A\,\in\,\GL_2(\Zh)$; we call it integral
when its entries lie in $\Z$. Since its determinant is then an integer
unit,
\begin{equation}\label{eq:integral-intersection}
 \GL_2(\Zh)\cap M_2(\Z)\ =\ \GL_2(\Z).
\end{equation}
Here $M_2(R)$ is the set of $2$ by $2$ matrices over a ring $R$.

\begin{lem}\label{lem:scalar}
If $a,\,b\,\in\,\Z\setminus\{0\}$ and $u\,\in\,\Zh^\times$ satisfy $a\,=\,bu$,
then $a\,=\,\epsilon b$ and $u\,=\,\epsilon$ for $\epsilon\,\in\,\{1,\,-1\}$.
Consequently, if $A,\,B$ are integral matrices with $B\,\ne\,0$ and
$A\,=\,\eta B$, where $\eta\,\in\,\Zh^\times$, then $\eta\,=\,\pm1$.
\end{lem}

\begin{proof}
In each $\Z_p$, the $p$-adic valuations of $a,\,b$ agree because $u$ is a
unit: these valuations are the exponents of $p$ in $|a|,\,|b|$. Unique
factorization gives $|a|\,=\,|b|$. Cancel the nonzero integer
$b$ in $b(u-\epsilon)\,=\,0$. For matrices apply this to a nonzero entry.
\end{proof}

\begin{lem}\label{lem:integral-vectors}
If $A\,\in \,M_2(\Zh)$ and $C,\,D\,\in\, M_2(\Z)$, with
$\det C\,\ne\,0$, satisfy $AC\,=\,D$, then $A\,\in \,M_2(\Z)$. If $A$ is invertible over $\Zh$,
it belongs to $\GL_2(\Z)$.
\end{lem}

\begin{proof}
For $d\,=\,\det C$, multiplying by the integral adjugate gives
$dA\,=\,D\operatorname{adj}(C)\,\in\, M_2(\Z)$. Reduction modulo $|d|$ yields
$\Z\cap d\Zh\,=\,d\Z$. Each entry of $dA$ is therefore $d$ times an integer;
injectivity of multiplication by $d$ proves the assertion.
Using \eqref{eq:integral-intersection} invertibility follows.
\end{proof}

\section{Marked completions and finite data}\label{sec:finite-data}

Here a representation means a group homomorphism. A finite
presentation gives finite representation sets in a finite group by assigning images to generators and checking that all relations are satisfied. For a knot
diagram, the Wirtinger generators are meridians of its arcs and the
crossing relations are conjugations; assigning target elements obeying
these relations is a group colouring. Evaluating the preferred
longitude word supplies Eisermann's additional information.

For $P\,\in\,\Z[Q]$, write $\coeff_y(P)$ for the coefficient of $y$.
Multiplication in this ring is the bilinear extension of multiplication
in $Q$. Then
\begin{equation}\label{eq:polynomial-coefficient}
 \coeff_y(P_Q^x(K))\,=\,
 \#\{\rho:G_K\to Q\,\,\big\vert\,\,\rho(\mu_K)\,=\,x,\ \rho(\lambda_K)\,=\,y\}.
\end{equation}
Only $y\,\in\, C_Q(x)\cap[Q,\,Q]$ can occur, because the marks commute and the
longitude is a commutator element; here $C_Q(x)$ is the centralizer
of $x$. Summing the coefficients forgets the
longitude distribution. Changing the boundary path is of no concern:
precomposing the representation with the inverse inner automorphism gives a bijection preserving their prescribed meridian and longitude images.

For a topologically finitely generated profinite group $P$ with ordered
string $\mathbf a\,=\,(a_1,\,\cdots,\,a_k)$, let $\mathcal Q(P,\mathbf a)$ be the
set of finite marked quotient classes. For $\mathbf x\,\in\, Q^k$, define
\[
 \begin{split}
 H_P(Q;\mathbf x)&\,=\,\#\{f:P\to Q\text{ continuous}\,\,\big\vert\,\,f(\mathbf a)=\mathbf x\},\\
 E_P(Q;\mathbf x)&\,=\,\#\{f:P\onto Q\text{ continuous}\,\,\big\vert\,\,f(\mathbf a)=\mathbf x\}.
 \end{split}
\]
These counts are finite. A common conjugation of all marks can be
absorbed into the isomorphism of the group.

\begin{prop}\label{prop:finite}
For topologically finitely generated profinite groups $(P,\,\mathbf a)$
and $(R,\,\mathbf c)$ with the same number of marks, the following three
statements are equivalent:
\begin{enumerate}
\item There is a continuous isomorphism $P\, \longrightarrow\, R$ taking
$\mathbf a$ to $\mathbf c$;

\item $\mathcal Q(P,\,\mathbf a)\,=\,\mathcal Q(R,\,\mathbf c)$;

\item $H_P(Q;\,\mathbf x)\,=\,H_R(Q;\,\mathbf x)$ for every finite $Q$ and ordered string
$\mathbf x$.
\end{enumerate}
\end{prop}

The unmarked principle is classical \cite{DFPR}. We include the marked
compactness argument because no transition maps are part of (ii).

\begin{proof}
Precomposition proves that (i) implies (ii) and (iii). It is clear that
\begin{equation}\label{eq:epi-recursion}
 E_P(Q;\,\mathbf x)\ =\ H_P(Q;\,\mathbf x)
 -\sum_{\substack{S<Q\\x_1,\ldots,x_k\in S}}E_P(S;\,\mathbf x);
\end{equation}
the sum is over actual subgroups, including different embeddings of
isomorphic subgroups. Induction on $|Q|$ recovers the epimorphism counts.
A marked quotient occurs precisely when its epimorphism count is
positive, thus proving that (iii) implies (ii).

Assume that (ii) holds. For each open normal subgroup $U\,\trianglelefteq_o\,R$, the set
$\mathcal E_U$ of marked epimorphisms $P\,\onto \,R/U$ is finite and nonempty.
In the compact product $\prod_U\mathcal E_U$, impose compatibility after
projection whenever $V\,\subseteq\, U$. These closed conditions have the
finite intersection property: for finitely many indices $U_i$, a
marked epimorphism to $R/\bigcap_iU_i$ satisfies all listed conditions
after projection. Compactness supplies a compatible family and a
continuous marked morphism $f\,:\,P\,\longrightarrow\, R$. Its compact image is closed and
projects onto every finite quotient, so it is all of $R$.
Interchanging $P,R$ gives a marked surjection $g\,:\,R\,\onto\, P$.

A topologically finitely generated profinite group is Hopfian: every
continuous surjective endomorphism is injective. Indeed, for
$\theta\,:\,P\,\onto\, P$, pullback permutes the finite set of open normal
subgroups of index at most $n$. Their intersection $U_n$ satisfies
$\theta^{-1}U_n\,=\,U_n$, and hence $\ker\theta\,\subseteq\, U_n$. The intersection
of all $U_n$ is trivial. Apply this to $gf$; injectivity of $gf$ gives
injectivity of $f$. The resulting continuous bijection is a
homeomorphism by compactness, thus proving that (ii) implies (i).
\end{proof}

For one mark the proposition identifies $\Wk K$, $\Qwk K$ and all
$N_Q(z;K)$; for two marks it identifies $\Pk K$, $\Qk K$ and all
coefficients in \eqref{eq:polynomial-coefficient}. Also
\begin{equation}\label{eq:one-count-from-polynomial}
 N_Q(z;K)\ =\ \sum_{\substack{x,y\in Q\\xy=z}}\coeff_y(P_Q^x(K)).
\end{equation}
This sum can loose information in a single target group. Across all finite groups, the dictionary (Proposition~\ref{prop:finite}) recovers the one-element completion, and
Theorem~\ref{thm:one-recovery} recovers the pair.

\begin{cor}\label{cor:finite-witness-algebraic}
If the one-element or ordered marked completions of two knots are
nonisomorphic, some finite marked quotient occurs for only one of them,
and some prescribed-image representation count differs.
\end{cor}
\begin{proof}
Take contrapositives in Proposition~\ref{prop:finite}.
\end{proof}

\section{JSJ structure and rooted restrictions}\label{sec:jsj}

\subsection{Pieces and root types}\label{subsec:hyperbolic-background}

Cutting along a two-sided torus in a three manifold produces two boundary copies. A piece
is a connected compact manifold remaining after such cuts; it may have
several boundary tori and need not be a knot exterior. The boundary copies created by cutting the JSJ tori are internal boundaries; the original torus $\partial E_K$ is the exterior boundary. A manifold is atoroidal if every incompressible torus is boundary-parallel.
The JSJ theorem gives a canonical minimal family of essential tori
cutting an irreducible orientable manifold with incompressible toral
boundary into Seifert or atoroidal pieces
\cite{JacoShalen,Johannson}. Hyperbolization makes the non-Seifert
pieces of a knot exterior hyperbolic \cite{ThurstonSurvey,ThurstonHaken};
see also geometrization \cite{PerelmanEntropy,PerelmanSurgery} and the
modern account \cite[Chapter~1]{AFW}.

A hyperbolic piece is a compact core of a complete finite-volume
hyperbolic $3$-manifold, obtained by truncating its torus cusps.
A cusp is an end homeomorphic to $T^2\times[0,\infty)$; its fundamental group is a
rank-two peripheral subgroup. The interior group is a torsion-free
nonuniform lattice in $\PSL_2(\C)$. Mostow--Prasad rigidity theorem realizes a
discrete isomorphism of such groups by one isometry of interiors, up to
a basepoint conjugation, and then by a homeomorphism of compact cores
\cite{Mostow,Prasad}. Realizing a given profinite map is a separate
role of Theorem~\ref{thm:xu-genuineness}.

A Seifert fibration decomposes a manifold into circle fibres with
fibred solid-torus neighbourhoods. A regular fibre has multiplicity one;
an exceptional fibre of multiplicity $m$ gives a cone point of order
$m$ in the quotient surface, called the base orbifold. More precisely,
if $B$ is the underlying surface, and $p_j$ are the cone points of
orders $m_j$, set $B^\circ\,=\,B\setminus\{p_j\}$. Then
\[
 \pi_1^{\mathrm{orb}}(B)\ \,=\ \,
 \pi_1(B^\circ)/\langle\!\langle a_j^{m_j}:j\rangle\!\rangle,
\]
where $a_j$ is a based loop around the removed point $p_j$.
The orbifold Euler characteristic is
$\chi(B)-\sum_j(1-1/m_j)$. The fibre subgroup is the kernel of the map
to this orbifold group. It is infinite cyclic and central when the
base is orientable, as it is for all the pieces below.

In the meridian--longitude basis $(m,\,\ell)$ of an unknotted solid torus,
a primitive curve $p[m]+q[\ell]$ is the torus knot $T(p,q)$; it is
nontrivial when $|p|,\,|q|\,\ge\,2$. For a companion $J$, pushing the curve
$p[m_J]+q[\lambda_J]$ into $\nu J$ gives a $(p,\,q)$ cable, with
$\gcd(p,\,q)\,=\,1$ and $|q|\,\ge\,2$. Removing its neighbourhood from $\nu J$
gives a cable space. A nontrivial prime knot is a satellite when its exterior contains an
essential torus, equivalently an essential companion torus in a
nontrivial pattern construction. A pattern is a knot inside a solid
torus; placing it in a companion neighbourhood need not make that torus
essential.
Cables of nontrivial companions are satellites. Cables of the unknot
are torus knots, and can be unknotted when $|p|\,=\,1$.

The Seifert pieces in a minimal knot-exterior JSJ decomposition are
as follows (see \cite[Theorem~4.18]{Budney}):
\begin{enumerate}
\item torus-knot exteriors, with base a disc with two cone points;
\item cable spaces, with base an annulus with one cone point;
\item composing spaces $\mathcal C_r\,=\,\Sigma_{0,r+1}\times S^1$,
where $\Sigma_{0,r+1}$ is a disc with $r$ holes and $r\,\ge\,2$.
\end{enumerate}
All have negative orbifold Euler characteristic and are called major
Seifert pieces. Thickened tori and twisted interval bundles over the
Klein bottle are absent. The dual graph has a vertex for each piece and
an edge for each cut torus. Every such torus separates in $S^3$, so this
graph is a finite tree. Its root contains $T_K$.

For a nontrivial prime knot, the root is a torus-knot exterior, a cable
space, or a hyperbolic piece. A torus-knot exterior has no children and
a cable root has one. A composite knot has a composing root, with one
child branch for each prime summand. Each branch is the full summand
exterior, possibly with its own satellite decomposition
\cite[Proposition~4.17, Theorem~4.18 and Corollary~4.19]{Budney}.
Thus nontrivial prime knots are torus, hyperbolic, or satellite knots.

A graph-manifold has only Seifert pieces; a graph knot has such an
exterior. A mixed exterior has both types. A maximal Seifert block is
the union of a maximal connected subtree of Seifert vertices. Its
boundary consists of interfaces with hyperbolic pieces and the exterior
torus if present. Seifert--Seifert interfaces lie inside the block.

\subsection{Fibre recognition and the outer slope}

A characteristic subgroup is preserved by every automorphism. A normal
subgroup is virtually central when its centralizer has finite index.
For a major Seifert group completion, the fibre closure is the unique
maximal virtually central normal procyclic subgroup
\cite[Theorem~5.5(1)]{WilkesSeifert} (see also
\cite[Proposition~8.2(1)]{Xu25}). It is therefore transported by every
continuous isomorphism.

\begin{lem}\label{lem:fibre-slope}
For a nontrivial prime knot with Seifert root, a regular fibre on the
exterior boundary, oriented to follow the knot positively, has elements
\begin{equation}\label{eq:prime-fibre}
 h_K\ =\ \mu_K^{c_K}\lambda_K,\quad c_K\ \in\ \Z,\quad |c_K|\ \ge\ 2.
\end{equation}
Its closure is characteristic in the root completion and is a copy of
$\Zh$.
\end{lem}
\begin{proof}
The root is a torus-knot exterior or cable space. On the boundary around
the knot, a neighbouring regular fibre is a push-off following the knot
once, so its longitude coefficient is $+1$. Its meridian coefficient
is the linking number in this fibre framing. In the fixed cable
convention this is $pq$: linking between two parallel torus levels is
bilinear, with the meridian--meridian and zero-framed
longitude--longitude terms zero and the cross terms $1,\,0$.
Expanding $pm+q\ell$ gives $pq$ (see
\cite[Proposition~3.8]{Budney}). Since $p\,\ne\,0$ and $|q|\,\ge\,2$ for a cable,
and both absolute values are at least $2$ for a nontrivial torus knot,
we have $|c_K|\,\ge\,2$. Here the fibre slope is computed on the boundary of a neighbourhood of the cable knot. The above fibre-recognition  result shows that its closure is characteristic. The primitive vector $(c_K,\,1)$ in $\Zh^2$ identifies its closure
with $\Zh$.
\end{proof}

At a composing root the fibre is instead $\mu_K$. Hence, for any fixed
integer $n\,\ne\,0$, the element $\mu_K\lambda_K^n$ lies outside every
Seifert root fibre closure. At a prime root, membership would require
$(1,\,n)\,=u\,(c_K,1)$, giving $u\,=\,n$ and $1\,=\,nc_K$, which is impossible.
At a composing root the nonzero longitude coordinate suffices.

\subsection{Completed graphs of groups}

Piece groups $G_v$, torus groups $G_e$, and incidence injections
$\iota_{e,v}\,:\,G_e\,\into\, G_v$ form a graph of groups. For a tree its
fundamental group is the free product of vertex groups subject to
$\iota_{e,u}(z)\,=\,\iota_{e,v}(z)$ at each edge $z$. Van Kampen identifies it
with the manifold group; see \cite{SerreTrees}. The two-vertex case is
an amalgamated free product. The analogous profinite construction uses
free profinite products and closed normal relations, and is universal
for continuous maps. Consequently
\[
 \widehat{A*_CB}\ \cong\ \widehat A\amalg_{\widehat C}\widehat B.
\]
This pushout identity follows by comparing maps to every finite group.

The JSJ graph is efficient: its whole group is residually finite,
vertex and edge groups are separable, and their induced topologies are
full. Thus completed groups embed, and the profinite fundamental group
is the entire completion \cite[Theorem~A]{WZGraph},
\cite[Proposition~5.2]{Xu25}. The same applies to connected unions of
pieces. For a connected subtree $\Lambda$, write $M_\Lambda$ for this
union with its internal edges glued; edges to the complement become
boundary tori.

For a proper completed graph, the standard profinite tree has vertices
$\coprod_vP/\widehat{G_v}$ and edges $\coprod_eP/\widehat{G_e}$,
where $P$ is the completion. The left action has stabilizers
$g\widehat{G_v}g^{-1}$ and $g\widehat{G_e}g^{-1}$.
A profinite graph has compact totally disconnected vertex and edge
spaces and continuous incidence maps. Its tree property is exactness of
the augmented completed chain sequence
\[
 0\,\longrightarrow\,\Zh[[E]]\,\xrightarrow{\,\,\,\partial\,\,\,}\,\Zh[[V]]
 \,\xrightarrow{\,\,\,\varepsilon\,\,\,}\Zh\,\longrightarrow\,0.
\]
Here $\Zh[[X]]$ is the inverse limit of free modules on finite continuous
quotients of $X$, $\partial$ takes an edge to its terminal minus initial
vertex, and $\varepsilon$ sums the coefficients. This is a homological tree
notion; arbitrary vertices need not be joined by an ordinary finite
edge path \cite{RibesGraphs}.

\subsection{Peripheral maps and congruence}

The full peripheral structure comprises all boundary subgroup conjugacy
classes, including original and cut boundaries. Congruent completed
graph data consist of a graph isomorphism $\theta$, vertex and edge
isomorphisms $f_v,\,f_e$, and $h_{e,v}\in\widehat{G'_{\theta(v)}}$ with
\begin{equation}\label{eq:edge-congruence}
 f_v\widehat\iota_{e,v}\
 =\ \Inn(h_{e,v})\widehat\iota'_{\theta(e),\theta(v)}f_e.
\end{equation}
The same $f_e$ occurs at both ends of an edge.

For a subgroup $Q\,\le\, P$, the normalizer is
$N_P(Q)\,=\, \{g\,\,\big\vert\,\, gQg^{-1}\,=\,Q\}$. For the pieces, connected blocks, and
for nontrivial knot exteriors here, distinct boundary closures are
nonconjugate and
\begin{equation}\label{eq:normalizer}
 N_{\widehat{\pi_1(X)}}(\overline{\pi_1(T)})\ =\ \overline{\pi_1(T)}.
\end{equation}
These assertions follow from \cite[Proposition~6.7]{Xu25}, since the manifolds considered here are irreducible and neither a thickened torus nor the orientable $I$--bundle over the Klein bottle. For hyperbolic pieces, they also follow from the stronger malnormality statement in \cite[Lemma~6.2]{Xu25}.

If $\Inn(a)f(\overline P)\,=\,\overline Q$, define the normalized boundary
map as $f_T\,=\,\Inn(a)f|_{\overline P}$. Any other carrying element differs
by an element normalizing $\overline Q$, which by
\eqref{eq:normalizer} belongs to the abelian group $\overline Q$ and acts
trivially. Thus, $f_T$ is well defined.

\subsection{Rooted restrictions}

Wilton--Zalesskii establish efficiency and profinite detection of
geometric types and decompositions \cite{WZGraph,WZGeometries,WZ};
Wilkes supplies the compact boundary setting \cite{WilkesJSJ}.
We use the following consequences in Xu's formulation.

\begin{prop}\label{prop:jsj}
Let $M,\,N$ be nontrivial knot exteriors, and let
$f\,:\,\widehat{\pi_1(M)}\,\longrightarrow\,\widehat{\pi_1(N)}$
be a continuous isomorphism carrying the closure of the exterior peripheral subgroup of $M$ to a conjugate of that of $N$. Denote their JSJ trees respectively by $\Gamma_M,\, \Gamma_N$, rooted at the
vertices containing $\partial M,\,\partial N$. For a connected subtree $\Lambda\,\subseteq\,\Gamma_M$, let
$M_\Lambda$ be the union of its vertex pieces, glued along the edges of $\Lambda$.
Then the following three hold:
\begin{enumerate}
\item The isomorphism $f$ induces an isomorphism of rooted trees
$\theta\,:\,\Gamma_M\,\longrightarrow\,\Gamma_N$ and congruent data
$(\theta,(f_v),\,(f_e),\,(h_{e,v}))$ satisfying \eqref{eq:edge-congruence}.
The vertex correspondence preserves hyperbolic and Seifert
types. Each $f_v$ preserves the full peripheral structure
of its piece, including both boundaries arising from
JSJ tori and any original exterior boundary.

\item For every connected subtree $\Lambda$ with
$\Lambda'\,=\,\theta(\Lambda)$, the restricted congruent data
induce a continuous isomorphism
\[
 f_\Lambda\ :\ \widehat{\pi_1(M_\Lambda})\
 \longrightarrow\ \widehat{\pi_1(N_{\Lambda'})}.
\]
This isomorphism preserves the conjugacy classes of all boundary
subgroups of the block. These boundaries consist of the original
exterior boundary, when the root belongs to $\Lambda$, and the
tori corresponding to edges with exactly one endpoint in
$\Lambda$.

The map $f_\Lambda$ is obtained from the restricted graph data;
it is not asserted to be the literal restriction of $f$ to an
arbitrarily based block subgroup.

\item Let $\Lambda,\,\Upsilon$ be disjoint connected subtrees joined
by an edge $e$. Write
\[
 j_e\ :\ T_{e,\Lambda}\ \longrightarrow\ T_{e,\Upsilon}
\]
for the source gluing map between the two boundary copies
of the corresponding JSJ torus, and write $j'_{\theta(e)}$
for the corresponding target gluing map.
Let $A_{\Lambda,e}$ and $A_{\Upsilon,e}$ be the normalized
boundary maps induced by $f_\Lambda$ and $f_\Upsilon$ respectively.
Then
\begin{equation}\label{eq:block-boundary-compatibility}
 A_{\Upsilon,e}\circ\widehat{(j_e)_*}\ =\
 \widehat{(j'_{\theta(e)})_*}\circ A_{\Lambda,e}.
\end{equation}
Thus the two maps agree under the source and target
gluing identifications.

If $\Lambda$ contains the root, the normalized boundary
map of $f_\Lambda$ on $\partial M$ coincides with the normalized
boundary map of $f$ on $\partial M$, as a map
\[
 \widehat{\pi_1(\partial M)}\ \longrightarrow\
 \widehat{\pi_1(\partial N)}.
\]
All normalized maps are understood using fixed peripheral
representatives and the associated boundary-group
identifications.
\end{enumerate}
\end{prop}

\begin{proof}
Nontrivial knot exteriors are compact, orientable and
irreducible, with nonempty incompressible toral boundary.
Their JSJ graphs are efficient by
\cite[Proposition~5.2]{Xu25}.
Consequently, vertex and edge completions embed, and the
profinite fundamental group of the completed JSJ graph
identifies with the completion of the manifold group.

By \cite[Theorem~5.5]{Xu25}, the given isomorphism $f$
induces congruent completed graph data preserving the
hyperbolic and Seifert vertex types. The excluded case
of closed $\mathrm{Sol}$-manifolds does not arise here.
Since $f$ preserves the exterior peripheral conjugacy
class, \cite[Lemma~6.8]{Xu25} places the target exterior
boundary in the vertex corresponding to the source root.
Hence $\theta$ preserves the roots.
Moreover, \cite[Corollary~6.11]{Xu25} gives preservation
of the full peripheral structure at every vertex.
This proves~(i).

Restrict the congruent data to $\Lambda$ and $\Lambda'$.
By \cite[Proposition~3.11]{Xu25}, these data induce an
isomorphism of the corresponding profinite fundamental
groups. The connected blocks are themselves irreducible
with incompressible toral boundary, and efficiency
identifies these groups with
$\widehat{\pi_1(M_\Lambda)}$ and
$\widehat{\pi_1(N_{\Lambda'})}$ respectively.

On each vertex, the induced block map agrees with $f_v$
up to conjugation in the target block.
An original exterior boundary is therefore matched by
the vertex peripheral correspondence. A boundary exposed
by cutting an edge leaving $\Lambda$ is matched by
\eqref{eq:edge-congruence}. These account for every
boundary component of $M_\Lambda$, thus proving~(ii).

We next explain the pointwise compatibility.
For an edge $e$ incident to $v$, write
\[
 B_{e,v}\ =\ \widehat{\iota}_{e,v}(\widehat G_e)
 \subseteq\widehat G_v,
\]
and use primed notation for the target.
The congruence identity is
\[
 f_v\widehat{\iota}_{e,v}\ =\
 \operatorname{Inn}(h_{e,v})
 \widehat{\iota}'_{\theta(e),\theta(v)}f_e.
\]
Consequently, the normalized vertex boundary map is
\[
 \left.
 \operatorname{Inn}(h_{e,v}^{-1})f_v
 \right|_{B_{e,v}}\ =\
 \widehat{\iota}'_{\theta(e),\theta(v)}
 f_e\widehat{\iota}_{e,v}^{-1}.
\]
In particular, after identifying either boundary copy
with the abstract edge group, its normalized map is
the same edge isomorphism $f_e$.

The conjugations introduced when passing from $f_v$ to
a block map can be absorbed into the boundary carrying
element. Any two carrying elements differ by an element
of the normalizer of the target boundary subgroup.
By \eqref{eq:normalizer}, that normalizer is the boundary
subgroup itself. Since this group is abelian, the
difference acts trivially by conjugation.
Thus the normalized block boundary map coincides with the
normalized vertex boundary map.
At the two ends of $e$, both these maps are represented
by $f_e$, which proves
\eqref{eq:block-boundary-compatibility}.

Finally, suppose that $\Lambda$ contains the root.
Both the global map $f$ and the block map $f_\Lambda$
agree with the root vertex map up to conjugation in
their respective target groups. Normalize the root
vertex map using a carrying element inside the target
root group, as supplied by
\cite[Lemma~6.8]{Xu25}.
Absorbing the additional conjugations gives this same
normalized boundary map for both $f$ and $f_\Lambda$.
Self-normalization makes it independent of the carrying
choices. Hence the two exterior boundary maps agree
on every element of $\widehat{\pi_1(\partial M)}$.
\end{proof}

\begin{lem}\label{lem:normalization}
In the setting of Proposition~\ref{prop:jsj}, let $V \leq \widehat{\pi_1(M)} \text{ and }\, W \leq \widehat{\pi_1(N)}$ be the completed vertex groups of the root pieces, with representatives chosen to contain the exterior peripheral closures $P$ and $Q$, respectively. Normalize $f$ by one inner automorphism so that $f(P)\,=\,Q$. Then $f(V)\,=\,W$, and its restriction
$\psi\,=\,f\big\vert_V \,: \,V \,\longrightarrow \,W$  is a continuous isomorphism whose boundary restriction is $\psi\vert_P\, = \,f\vert_P$. In particular, if the normalized $f$ preserves
the exact meridian and longitude, so does $\psi$.
\end{lem}

\begin{proof}
The JSJ correspondence gives $f\big\vert_V\,=\,\Inn(g)f_v$, where
$f_v\,:\,V\,\longrightarrow\, W$ is an isomorphism and $g\,\in\,\widehat{\pi_1(N)}$.
Using \cite[Lemma~6.8]{Xu25} choose $b\,\in\, W$ with
$\Inn(b)f_v(P)\,=\,Q$. Both $\Inn(g)$ and $\Inn(b)$ carry
$f_v(P)$ onto $Q$, so $gb^{-1}$ normalizes $Q$ in the
target completion. Formula~\eqref{eq:normalizer} gives
$gb^{-1}\,\in\, Q\,\subset\, W$, hence $g\,\in\, W$. Therefore,
$f(V)\,=\,gWg^{-1}\,=\,W$. Moreover, $gb^{-1}$ centralizes $Q$, so
\[
 \Inn(g)f_v\big\vert_P\ =\ \Inn(b)f_v\big\vert_P.
\]
The literal root restriction consequently retains both
individual marks using the same conjugation. If the marks
were initially given up to simultaneous conjugacy, make that
one global conjugation prior to the calculation.
\end{proof}

\section{Prime branches and longitude coherence}\label{sec:constituents}

\subsection{The composing root}

For $K\,=\,K_1\#\cdots\#K_r$, with $r\,\ge\,2$, the canonical composing region is
$\mathcal C_r\,=\,\Sigma_{0,r+1}\times S^1$. Its outer boundary is $T_K$;
its inner tori $T_i$ attach the full prime exteriors $E_{K_i}$.
For $r\,=\,2$, the base is a pair of pants, and an arc between its boundary
circles times $S^1$ gives a connected-sum annulus. Nested composing
regions merge into the single canonical root
\cite[Proposition~4.17 and Corollary~4.19]{Budney}.

Choose inner loops $x_i$ and their paths so that
$d\,=\,x_1\cdots x_r$ is the outer preferred longitude. The loop directions
are the knot-longitude directions; the induced surface-boundary
orientation would invert the outer circle. Let $t$ be the positive
circle fibre. Then
\begin{equation}\label{eq:composing-discrete}
 \pi_1(\mathcal C_r)\ =\ F_r\times\langle t\rangle,
 \qquad F_r\ =\ \langle x_1,\,\cdots,\,x_r\rangle,
\end{equation}
where $F_r$ is the free group. The boundary groups are
$P_0\,=\,\langle t,\,d\rangle$ and $P_i\,=\,\langle t,\,x_i\rangle$.
On attachment, $(t,\,x_i)\,=\,(\mu_{K_i},\,\lambda_{K_i})$; on the outer torus,
$(t,\,d)\,=\,(\mu_K,\,\lambda_K)$.

Completion gives
\begin{equation}\label{eq:composing-completed}
 V_r\ =\ \widehat{\pi_1(\mathcal C_r)}\ =\ \widehat F_r\times\overline{\langle t\rangle},
 \qquad\overline{\langle t\rangle}\ \cong\ \Zh.
\end{equation}
The groups $\overline{P_i}$, $0\,\le\, i\,\le\, r$, are copies of $\Zh^2$.
The restriction of global knot abelianization is
\begin{equation}\label{eq:composing-chi}
 \chi\,:\,V_r\,\longrightarrow\,\Zh,\qquad\chi(t)\,=\,1,\qquad\chi(x_i)\,=\,0.
\end{equation}
We omit its hat in this section.

The Van Kampen theorem along the connected-sum annuli also gives
\begin{equation}\label{eq:meridional-amalgam}
 G_K\,=\,G_{K_1}*_{\langle\mu\rangle}\cdots*_{\langle\mu\rangle}G_{K_r},
 \qquad\lambda_K\,=\,\lambda_{K_1}\cdots\lambda_{K_r}
\end{equation}
\cite[Proposition~2.5]{Eisermann}. Each factor has a retraction: keep
it fixed and map every other factor by abelianization to the common
meridian. Completing the retraction embeds its completion in the whole
group. Under a profinite isomorphism preserving the exterior peripheral structure, Proposition~\ref{prop:jsj} identifies the summand factors through their canonical rooted JSJ branches.

\subsection{Eliminating meridional shears}

Boundary subgroups alone allow shears. In $\widehat F_2\times\Zh$, for
$a\,\in\,\Zh$ the automorphism given by
\begin{equation}\label{eq:shear-example}
 t\,\longmapsto\, t,\qquad x_1\,\longmapsto\, x_1t^a,\qquad x_2\,\longmapsto\, x_2t^{-a}
\end{equation}
preserves inner boundary groups and fixes $d\,=\,x_1x_2$. For $a\ne0$ it
fails to preserve $\chi$. A restriction of a global knot-group
isomorphism has the additional compatibility condition
\eqref{eq:chi-compatible}, which excludes these shears.

\begin{lem}\label{lem:coherence}
Let $V_r\,=\,\widehat F_r\times\overline{\langle t\rangle}$
and $V_r'\,=\,\widehat F_r'\times\overline{\langle t'\rangle}$
be composing-space groups, with inner generators
$x_i$ and $y_j$, outer words
$d\,=\,x_1\cdots x_r$ and $d'\,=\,y_1\cdots y_r$, and
normalized maps $\chi,\,\chi'$ as in
\eqref{eq:composing-chi}. Suppose a continuous
isomorphism $\psi\,:\,V_r\,\longrightarrow\, V_r'$ satisfies the following:
\begin{enumerate}
\item $\psi(t)\,=\,t'^\alpha$, with $\alpha\,\in\,\Zh^\times$;

\item $\psi$ permutes the conjugacy classes of the inner boundary subgroups;

\item $\chi'\psi\ =\ \alpha\chi$;

\item $\psi(d)$ is conjugate to $d'^\beta$, with $\beta\,\in\,\Zh^\times$.
\end{enumerate}
Then there are a permutation $\sigma$ and elements
$b_i\,\in\, V_r'$ such that
\[
 \psi(x_i)\ =\ b_i y_{\sigma(i)}^\beta b_i^{-1}.
\]
The pair $(t,\,x_i)$ is therefore carried to
$(t'^\alpha,\, y_{\sigma(i)}^\beta)$ with one
common conjugator for each $i$.
\end{lem}

\begin{proof}
By (ii), choose a permutation $\sigma$ and elements $b_i$ with
\[
 \psi(\overline{P_i})\ =\ b_i\overline{P'_{\sigma(i)}}b_i^{-1}.
\]
This fixes the direction of the conjugation: the normalized
boundary map is $\Inn(b_i^{-1})\psi|_{\overline{P_i}}$.
The target boundary group is abelian and is
freely generated as a $\Zh$-module by $t'$
and $y_{\sigma(i)}$. Consequently
\begin{equation}\label{eq:inner-general-form}
 \psi(x_i)\ =\ b_i t'^{\rho_i} y_{\sigma(i)}^{u_i}b_i^{-1}
 \qquad(\rho_i,\,u_i\ \in\ \Zh).
\end{equation}
The image of $t$ is $t'^\alpha$ and is central,
so after this same carrying conjugation the
boundary map has matrix
\[
 \begin{pmatrix}\alpha&\rho_i\\0&u_i\end{pmatrix}.
\]
It is invertible over $\Zh$, and $\alpha$ is
a unit. Its determinant $\alpha u_i$ is a
unit, so $u_i$ is a unit as well.

Apply $\chi'$ to \eqref{eq:inner-general-form}.
Conjugation disappears under a homomorphism to
an abelian group, and $\chi'(y_j)\,=\,0$. Thus
\[
 \rho_i\ =\ \chi'\psi(x_i)\ =\ \alpha\chi(x_i)\ =\ 0.
\]

Next quotient by the central fibre and abelianize
the free profinite base. Its abelianization is
$\Zh^r$ with basis $[y_1],\,\cdots,\,[y_r]$.
All $b_i$ disappear in this abelianization.
Applying (iv) to the product $d$ yields
\[
 \sum_{i=1}^r u_i[y_{\sigma(i)}]\ =\ \beta\sum_{j=1}^r[y_j].
\]
Independence of the basis vectors gives
$u_i\,=\,\beta$ for every $i$. Finally $t'$ is
central, so each $b_i$ also conjugates
$t'^\alpha$ to itself. The same $b_i$ therefore
works for both members of the ordered pair.
\end{proof}

\subsection{Proof of Theorem C of introduction}

We now prove the common-exponent theorem stated
in the introduction.

\begin{proof}[Proof of Theorem~\ref{thm:constituents}]
First make the one global conjugation which
turns the two assumed peripheral equalities
into exact equalities. Since $\alpha,\,\beta$
are units, these equalities also identify
the whole exterior peripheral closure.
Proposition~\ref{prop:jsj} gives a rooted
JSJ correspondence.

The source root is the composing space $\mathcal C_r$,
with $r$ child edges and one external boundary.
The target root is Seifert, has the same
boundary count, and has at least two children.
Among the Seifert root types of knot exteriors,
only a composing space has that property.
The target knot is therefore composite, its root is $\mathcal C_s$,
and $s\,=\,r$. This also shows that being composite is detected
in the marked setting. If the source were prime and the target
composite, applying the same reasoning to the inverse would
give a contradiction.
The correspondence permutes the child edges;
call the resulting permutation $\sigma$.

By Lemma~\ref{lem:normalization}, the root
restriction $\psi$ has the exact actions
\[
 \psi(t)\ =\ t'^\alpha,\qquad \psi(d)\ =\ d'^\beta.
\]
It permutes all inner boundary subgroups.
The global completed abelianization satisfies
\eqref{eq:chi-compatible}, because the given
meridian image determines its coefficient.
Restricting that identity to the composing
root gives $\chi'\psi\,=\,\alpha\chi$.
All hypotheses of Lemma~\ref{lem:coherence}
are now verified. Its conclusion gives the
normalized map on every child torus:
\begin{equation}\label{eq:branch-boundary-map}
 t\ \longmapsto\ t'^\alpha,
 \qquad x_i\ \longmapsto\ y_{\sigma(i)}^\beta.
\end{equation}

Remove the root vertex from the JSJ tree.
The component attached at $T_i$ is the full
subtree of $E_{K_i}$, and its corresponding
component is the subtree of $E_{L_{\sigma(i)}}$.
Apply Proposition~\ref{prop:jsj}(ii) to these
connected subtrees. It gives an isomorphism
of their completed fundamental groups.
Part (iii) transfers the normalized boundary
map \eqref{eq:branch-boundary-map} from the
root side to the branch side. Under the
attachment conventions, these two curves
are exactly the meridian and preferred
longitude of the prime summand. This gives
the required $\Phi_i$ with the common
exponents $\alpha,\,\beta$. A carrying
conjugator is used for the entire boundary
torus, so its two equalities are simultaneous.
\end{proof}

\begin{cor}\label{cor:prime-profile}
An isomorphism of exact ordered marked
completions of composite knots induces
an equality of the multisets of the exact
ordered marked completions of their
nontrivial prime summands. It also matches
the full rooted JSJ branch of each summand.
\end{cor}

\begin{proof}
Set $\alpha\,=\,\beta\,=\,1$ in
Theorem~\ref{thm:constituents} and retain
the branch correspondence from its proof.
\end{proof}

A prime branch consisting of a single Seifert vertex is a torus-knot exterior; a single hyperbolic
vertex is a hyperbolic-knot exterior; a prime branch with internal JSJ
edges is a satellite exterior. This recovers the prime decomposition
and types before geometric recognition of the summands.

\section{Relative boundary realization}\label{sec:realization}

Relative reconstruction means realizing the boundary maps specified by
the profinite isomorphism. Knowing that corresponding pieces happen to
be homeomorphic is insufficient: their homeomorphisms must agree on every common torus. We distinguish two arithmetic properties. A boundary map
$A\,\in\,\GL_2(\Zh)$ is unit-regular if $A\,=\,\eta B$, with
$\eta\,\in\,\Zh^\times$ and $B\,\in\,\GL_2(\Z)$. It is integral when
$A\,\in\,\GL_2(\Z)$, the stronger property called peripheral regularity
(see \cite[Definition~1.3]{XuRegular}). A homeomorphism with matrix $B$ realizes
$A$ with coefficient $\eta$ when $A\,=\,\eta B$.

\subsection{Peripheral regularity and Seifert scales}

\begin{thm}[{\cite{Xu25}}]\label{thm:hyp-regular}
For any continuous isomorphism between completions of cusped hyperbolic
piece groups, its normalized cusp maps have matrices $A_i\,=\,\eta B_i$,
where $\eta\,\in\,\Zh^\times$ is one common unit and $B_i\,\in\,\GL_2(\Z)$.
\end{thm}
This is \cite[Theorem~7.2(1)]{Xu25}, with cusp correspondence from its
Lemma~6.3. Liu's homological regularity and almost-rigidity work precedes
this boundary theorem \cite{LiuAlmost}.

\begin{thm}[{\cite{XuRegular}}]\label{thm:hyp-integral-regular}
For any continuous isomorphism between completions of cusped hyperbolic
piece groups, every normalized cusp map belongs to $\GL_2(\Z)$.
\end{thm}

This is \cite[Theorem~1.4]{XuRegular}; its proof uses the relative
cohomology framework of \cite{WilkesRelative}. It supplies integral
cusp maps. Their simultaneous realization by one homeomorphism of the
piece uses Theorem~\ref{thm:xu-genuineness} separately.

For a major Seifert piece, use a fibre--cross-section basis $(h,\,e)$ on
each boundary, with $e$ projecting once around its base boundary circle.
A peripheral-preserving completed isomorphism has matrix
\[
 \begin{pmatrix}\kappa_v&\rho_{v,T}\\0&\tau_v\end{pmatrix},
 \qquad\kappa_v,\tau_v\,\in\,\Zh^\times,\quad\rho_{v,T}\,\in\,\Zh.
\]
The diagonal units, modulo orientation signs, are independent of $T$;
$\kappa_v$ scales the fibre while $\tau_v$ the base boundary. Their pair
is the scale type. This is Wilkes's Seifert scale analysis
\cite{WilkesSeifert,WilkesGraphI} (see
\cite[Propositions~8.2 and 8.12]{Xu25}). Across a Seifert--Seifert edge,
\begin{equation}\label{eq:scale-propagation}
 \kappa_w\ =\ \pm\tau_v,\qquad\tau_w\ =\ \pm\kappa_v
\end{equation}
\cite[Lemma~9.3(2)]{Xu25}. Thus sign scales propagate through a connected
Seifert subtree; compatibility will also control the off-diagonal terms.

\begin{prop}\label{prop:automatic-regularity}
Let $\Phi\,:\,\widehat{G_K}\,\longrightarrow\,\widehat{G_L}$ be a continuous isomorphism of
nontrivial knot completions preserving the exterior peripheral subgroup
up to conjugacy. If its normalized exterior map $A_0$ is integral, then
all normalized boundary maps of its JSJ vertex maps are integral.
For the exact ordered marking, $A_0\,=\,I$.
\end{prop}

\begin{proof}
Use Proposition~\ref{prop:jsj}. Every hyperbolic boundary map is integral
by Theorem~\ref{thm:hyp-integral-regular}. An integral Seifert boundary
matrix is upper triangular in fibre bases with determinant $\pm1$, so
both diagonal scales are signs.

For an all-Seifert exterior, $A_0$ forces both scaling factors at the root to lie in $\{1,\,-1\}$, and
\eqref{eq:scale-propagation} propagates them through the tree. Otherwise
each maximal Seifert subtree meets a hyperbolic vertex. At such an
interface, source and target gluing matrices $J,\,J'\,\in\,\GL_2(\Z)$ give
$A_S\,=\,J'A_HJ^{-1}$, so the Seifert map is integral there. This seeds sign
scales in that subtree, which then propagate.

At a Seifert--Seifert edge the two fibre slopes are independent over
$\mathbb Q$: parallel fibres would join the fibrations, contradicting JSJ
minimality. Let $C,\,C'$ have these two source and target fibre vectors
as columns. The common edge map satisfies
\[
 A_eC\ =\ C'\diag(\epsilon_1,\,\epsilon_2),
 \quad\epsilon_i\ \in\ \{1,\,-1\},\quad\det C\ \ne\ 0.
\]
Lemma~\ref{lem:integral-vectors} makes $A_e$ integral and hence an element
of $\GL_2(\Z)$. This controls its shear entries as well as scales,
consistent with \cite[Lemma~9.3(4)]{Xu25}. All other edges have a
hyperbolic endpoint and integral gluing matrices. The only remaining
piece boundary is the exterior torus with map $A_0$.
\end{proof}

For a mixed exterior, \cite[Theorem~9.2]{XuRegular} already implies
exterior regularity from peripheral preservation: its requirement that
each Seifert piece contain at most one exterior boundary component is
automatic for a knot exterior. The proposition also records the internal
edge calculation. In a graph exterior the exact root marking supplies
the initial regular boundary.

\subsection{One common coefficient on each Seifert block}

\begin{thm}[{\cite{Xu25}}]
\label{thm:xu-boundary}
Let $X,\,Y$ be graph-manifolds with nonempty
incompressible toral boundary, as occur as
Seifert blocks in knot exteriors. Let
$f\,:\,\widehat{\pi_1(X)}\,\longrightarrow\,\widehat{\pi_1(Y)}$ be
a continuous isomorphism preserving the full
peripheral structure. Suppose the normalized
boundary maps of $f$ are unit-regular on
selected components $T_1,\,\cdots,\,T_k$, where
$1\,\leq\, k\,\leq\,\#\partial X$. Then there exist
a unit $\eta\,\in\,\Zh^\times$ and two
homeomorphisms $h_+,\,h_-\,:\,X\,\longrightarrow\, Y$ such that,
on every selected component,
\begin{equation}\label{eq:xu-common-scalar}
 A_i\ =\ \eta B_{+,i}\ =\ -\eta B_{-,i}.
\end{equation}
Here $A_i$ is the prescribed normalized
profinite boundary matrix and $B_{\pm,i}$
is the matrix of $h_\pm$ on that component.
Both homeomorphisms realize the prescribed correspondence of the
selected components. The scalar and the choice of sign are shared
by the entire selected set, not chosen separately for its members.
\end{thm}

For a block with multiple JSJ pieces this is
\cite[Theorem~9.11]{Xu25}. If the block consists of one
Seifert piece, the same simultaneous conclusion is
\cite[Proposition~8.19]{Xu25}. The exceptional boundary cases
excluded in Section~\ref{sec:jsj} ensure that the normalized boundary
maps are unambiguous.

\begin{prop}\label{prop:graph}
Let $X,\,Y$ be graph-manifold blocks as in
Theorem~\ref{thm:xu-boundary}, and let
$f\,:\,\widehat{\pi_1(X)}\to\widehat{\pi_1(Y)}$ be
a continuous isomorphism preserving full
peripheral structure. If its maps on a
nonempty selected set of boundary components
are integral, then one homeomorphism
$h\,:\,X\,\longrightarrow\,
Y$ realizes all these prescribed
boundary maps simultaneously, with
coefficient $1$.
\end{prop}

\begin{proof}
Integral maps are unit-regular, so apply
Theorem~\ref{thm:xu-boundary}. On one
selected torus, $A_i$ and $B_{+,i}$ both
belong to $\GL_2(\Z)$. Lemma~\ref{lem:scalar}
gives $\eta\,=\,\pm1$. If $\eta\,=\,1$, then choose
$h_+$. If $\eta\,=\,-1$, choose $h_-$ so
that $-\eta\,=\,1$. The same choice works
on every selected component because
\eqref{eq:xu-common-scalar} uses a single
scalar and consistent signs.
\end{proof}

\subsection{Hyperbolic genuineness and all cusp boundaries}

\begin{lem}\label{lem:hyperbolic}
Let $H,\,H'$ be compact orientable cores obtained by
truncating the cusps of noncompact complete finite-volume hyperbolic
$3$-manifolds. Let
$f\,:\,\widehat{\pi_1(H)}\,\longrightarrow\, \widehat{\pi_1(H')}$ be a continuous isomorphism
preserving full peripheral structure. Then one homeomorphism
$h\,:\,H\,\longrightarrow\, H'$ realizes all normalized boundary maps of $f$
simultaneously. In particular each such map is integral.
\end{lem}

\begin{proof}
The interior groups are torsion-free nonuniform lattices.
Theorem~\ref{thm:xu-genuineness} therefore gives
\[
 f\ =\ \Inn(g)\widehat\alpha,\quad
 \alpha\ :\ \pi_1(H)\ \xrightarrow{\,\,\,\cong\,\,\,}\ \pi_1(H'),\quad
 g\ \in\ \widehat{\pi_1(H')}.
\]
Mostow--Prasad rigidity theorem realizes the one discrete isomorphism $\alpha$ by
an isometry of the complete interiors, up to the basepoint
conjugation. Adjusting the cusp truncations in collars gives one
homeomorphism $h\,:\,H\,\longrightarrow\, H'$. We need to check its boundary maps agree
with the prescribed maps of $f$, not merely that it permutes the
same set of boundary components.

Fix a discrete source cusp group $P\,=\,\pi_1(T)$ and its matched
target cusp group $Q\,=\,\pi_1(T')$. Choose a carrying element $a$ so
that the normalized map is
\[
 f_T\ =\ \Inn(a)f\big\vert_{\overline P}\ :\ \overline P\ \longrightarrow\ \overline Q.
\]
The discrete isomorphism $\alpha$ takes $P$ to a discrete
conjugate of a cusp group. The target is $Q$: after taking
closures its target is conjugate to $\overline Q$ under $f$,
and distinct cusp closures are nonconjugate. Hence choose
$s\,\in\,\pi_1(H')$ with $\alpha(P)\,=\,sQs^{-1}$.

For $p\,\in\,\overline P$ we calculate
\begin{equation}\label{eq:cusp-carrying}
 f_T(p)\ =\ (ags)\bigl(s^{-1}\widehat\alpha(p)s\bigr)(ags)^{-1}.
\end{equation}
Both $f_T(\overline P)$ and
$s^{-1}\widehat\alpha(\overline P)s$ equal $\overline Q$.
Therefore, $ags$ normalizes $\overline Q$. Self-normalization
\eqref{eq:normalizer} puts $ags$ in $\overline Q$, and this
abelian group centralizes itself. Formula~\eqref{eq:cusp-carrying}
reduces to
\[
 f_T(p)\,\ =\,\ s^{-1}\widehat\alpha(p)s.
\]
This is precisely the completed boundary map of $h$ with the
corresponding path to $T'$. Its matrix belongs to $\GL_2(\Z)$.

Repeat this normalization for every cusp. The carrying elements
$a,\,s$ can depend on the cusp; $\alpha$ and its homeomorphism $h$
remain the same. Thus all cusp maps are realized by one
homeomorphism of the entire hyperbolic piece.
\end{proof}

\subsection{Interface maps and collar gluing}

Let a hyperbolic piece $H$ meet a maximal Seifert block $X$, with gluing
maps $j\,:\,T_H\,\longrightarrow\, T_X$ and $j'\,:\,T_{H'}\,\longrightarrow\, T_{X'}$. Their matrices are
$J,\,J'\,\in\,\GL_2(\Z)$. Congruent graph data and normalization give
\[
\begin{tikzcd}[column sep=large,row sep=large]
 \widehat{\pi_1(T_H)}\arrow[r,"J"]\arrow[d,"A_H"']
 &\widehat{\pi_1(T_X)}\arrow[d,"A_X"]\\
 \widehat{\pi_1(T_{H'})}\arrow[r,"J'"']
 &\widehat{\pi_1(T_{X'})}.
\end{tikzcd}
\]
Thus
\begin{equation}\label{eq:interface-integrality}
 A_XJ\ =\ J'A_H,\quad A_X\ =\ J'A_HJ^{-1}\ \in\ \GL_2(\Z).
\end{equation}
This transfers integrality from the hyperbolic side. The same diagram
with two hyperbolic sides records their compatibility. A
Seifert--Seifert torus lies inside one maximal block and is handled by
its single homeomorphism, rather than independent scalar choices.

\begin{lem}\label{lem:gluing}
Homeomorphisms of corresponding hyperbolic pieces and maximal Seifert
blocks that realize the normalized maps on all remaining interfaces
can be adjusted in disjoint boundary collars to glue. Their exterior
boundary action is retained.
\end{lem}

\begin{proof}
At an interface $j\,:\,T_U\,\longrightarrow\, T_V$, with target gluing $j'$, the required
pointwise equality is
\begin{equation}\label{eq:actual-gluing}
 h_V\big\vert_{T_V}\circ j\ =\ j'\circ h_U|_{T_U}.
\end{equation}
Graph compatibility gives equality of their profinite matrices, hence
of the integral matrices by the injectivity of the map $\Z\,\longrightarrow\,\Zh$. The two torus
maps are therefore isotopic.

Put $k_0\,=\,h_V\big\vert_{T_V}$ and $k_1\,=\,j'\circ h_U\big\vert_{T_U}\circ j^{-1}$.
The target homeomorphism $\delta\,=\,k_1k_0^{-1}$ is isotopic to the identity map.
Choose $\delta_t$ with $\delta_0\,=\,\delta$, $\delta_1\,=\,\id$. A collar is a
product neighbourhood of a boundary with an interval. On a target
collar $T_{V'}\times[0,\,\varepsilon]$, the map
\[
 (z,\,u)\ \longmapsto\ (\delta_{u/\varepsilon}(z),u)
\]
extends $\delta$ at the boundary and becomes the identity map at the inner
face. Extend by the identity map elsewhere and compose with $h_V$.
Its boundary map is now $k_1$, giving \eqref{eq:actual-gluing}.

Collars at distinct boundary components can be chosen to be disjoint, so
performing these adjustments at every interface preserves previous
agreements. No adjustment is needed at the exterior boundary. The maps
and their inverses descend to the glued manifolds. This is the
coefficient-$1$ gluing mechanism of
\cite[Lemma~6.13(2) and Corollary~6.17(2)]{Xu25}.
\end{proof}

\subsection{Reconstruction and ordered recognition}

\begin{thm}\label{thm:relative}
Let $M,\,N$ be exteriors of nontrivial knots, and let
$f\,:\,\widehat{\pi_1(M)}\,\longrightarrow\,\widehat{\pi_1(N)}$ be a continuous isomorphism
preserving the exterior peripheral subgroup up to conjugacy. If its
normalized exterior map is $A_0\,\in\,\GL_2(\Z)$, a homeomorphism
$h\,:\,M\,\longrightarrow\, N$ realizes that same boundary map.
\end{thm}

\begin{proof}
Proposition~\ref{prop:jsj} matches the rooted trees, piece types, and
full peripheral structures. Proposition~\ref{prop:automatic-regularity}
makes the required maps integral. If every piece is Seifert, take the
entire exterior as one block and select its boundary. The induced map
is a full peripheral isomorphism, so Proposition~\ref{prop:graph}
realizes $A_0$ directly.

Otherwise, apply Lemma~\ref{lem:hyperbolic} to each hyperbolic piece.
For every maximal Seifert block, restrict the congruent graph data as
in Proposition~\ref{prop:jsj}(ii). This supplies a block completion
isomorphism preserving every boundary. Each block boundary is a
hyperbolic interface or the exterior torus: another Seifert neighbour
would contradict maximality. Its map is integral by
\eqref{eq:interface-integrality} or is $A_0$ by
Proposition~\ref{prop:jsj}(iii). The selected set is nonempty, because a
proper connected subtree has an edge to its complement.

Select all block boundaries. Proposition~\ref{prop:graph} realizes them
by one homeomorphism with coefficient $1$. Thus a single sign choice
controls every interface of that block. The hyperbolic and block maps
are compatible on every remaining cut torus; Lemma~\ref{lem:gluing}
joins them. The exterior boundary lies on a hyperbolic piece or among
a block's selected boundaries, so its action remains $A_0$.
\end{proof}

Theorem \ref{thm:relative} realizes the prescribed boundary action, not necessarily
the entire profinite map. The isomorphism hypothesis is indispensable:
the retraction $G_{T\#T}\,\longrightarrow\, G_T$ that keeps one trefoil factor and
abelianizes the other preserves both peripheral elements and has
boundary matrix $I$, but the exteriors are not homeomorphic.

\begin{cor}\label{cor:ordered-recognition}
For oriented tame knots, $K\,\simeq\, L$ if and only if $\Pk K\,\cong\,\Pk L$.
\end{cor}

\begin{proof}
An oriented equivalence induces the marked isomorphism. Conversely,
normalize the two marks by one conjugation. If one knot is the unknot,
its longitude is trivial; then \eqref{eq:completed-periphery} implies that the
other is also the unknot. For nontrivial knots the marked exterior map
has matrix $I$, so Theorem~\ref{thm:relative} produces an exterior
homeomorphism with that boundary action.

The determinant $+1$ condition, in consistently signed boundary bases, forces
this homeomorphism to preserve the ambient orientation. Meridian
preservation permits extension over the filling solid torus. The
positive longitude coefficient makes its action on the core generator
of $H_1(\nu K;\,\Z)$ equal to $+1$, so the knot direction is preserved.
For graph knots the argument uses only graph-block realization and
filling.
\end{proof}

\section{Recovering the ordered marking from one element}
\label{sec:one-element}

The mark is the exact element in the entire completion. An arbitrary
rank-two torus automorphism need not recover its factors: the matrix
$\left(\begin{smallmatrix}1&0\\2&-1\end{smallmatrix}\right)$ fixes $(1,1)$
and changes the ordered pair. We use the root structure and filling
topology of the knot group. We prove the local steps for
$w_{K,n}\,=\,\mu_K\lambda_K^n$, with a fixed integer $n\,\ne\,0$.

\subsection{The original boundary-intersection inputs}

For an element $z$ of a group $P$, the centralizer is
$C_P(z)\,=\,\{a\,\in\, P\,\,\big\vert\,\, az\,=\,za\}$. This is a closed subgroup when $P$
is profinite. An isomorphism carries this subgroup onto the
centralizer of the image element.
We require the following two established local facts.

\begin{prop}\label{prop:local-boundary-intersections}
Let $V$ be the completion of the group of a JSJ piece of a
nontrivial knot exterior. Let $P_i$ denote its completed
boundary-torus subgroups with chosen representatives. Then
the following two hold:
\begin{enumerate}
\item If the piece is hyperbolic, the family $\{P_i\}$ is malnormal:
      $P_i\cap aP_ja^{-1}\,=\,1$ unless $i\,=\,j$ and $a\,\in\, P_i$.
\item If the piece is Seifert, let $Z_V$ be the closure of its
      regular-fibre subgroup. Unless $i\,=\,j$ and $a\,\in \,P_i$,
      \[
       P_i\cap aP_ja^{-1}\,\ =\,\ Z_V.
      \]
      The subgroup $Z_V\,\cong\,\Zh$ is contained in each boundary
      group and is characteristic in $V$.
\end{enumerate}
\end{prop}

\begin{proof}
The first assertion is
\cite[Lemma~4.5]{WZGeometries}, also stated as
\cite[Lemma~6.2]{Xu25}.

For the second assertion, the original
Seifert boundary-intersection argument is in
\cite[Section~5]{WZGraph}; the form allowing arbitrary completed
conjugators is \cite[Lemma~6.4]{WilkesGraphI}, restated in
\cite[Lemma~6.4]{Xu25}. Characteristic fibre recognition is
\cite[Theorem~5.5(1)]{WilkesSeifert}, or
\cite[Proposition~8.2(1)]{Xu25}.
All pieces here are major Seifert pieces, so the solid torus,
thickened torus, and twisted Klein-bottle bundle exceptions are
absent. The nonconjugacy and normalizer conventions are explained
in Section~\ref{sec:jsj}.
\end{proof}

\subsection{The marked element identifies the exterior torus}

\begin{lem}\label{lem:word-centralizer}
Let $K$ be a nontrivial oriented tame knot and $n\,\ne\,0$ an integer.
Then
\[
 C_{\widehat{G_K}}(\mu_K\lambda_K^n)\ =\ \overline{P_K}\ \cong\ \Zh^2.
\]
If the JSJ decomposition is nontrivial, this element fixes exactly
the root vertex in the standard profinite JSJ tree.
\end{lem}

\begin{proof}
Write $P\,=\,\overline{P_K}$, $w\,=\,\mu_K\lambda_K^n$, and let $V$
be the completed root group containing $P$. The element $w$
is nonidentity because its boundary coordinate $(1,\,n)$ is
nonzero. If the root is Seifert, $w$ does not belong to its
fibre closure: this was checked after
Lemma~\ref{lem:fibre-slope}. Explicitly, a composing fibre has
coordinates $(1,\,0)$, while a prime Seifert fibre has
coordinates $(c_K,\,1)$ with $|c_K|\,\ge\,2$; neither can have a
profinite multiple equal to $(1,\,n)$.

Suppose first that the JSJ tree has edges, and let $v$ be the
vertex with stabilizer $V$. Since $w\,\in\, P\,\le\, V$, it fixes $v$.
We show that $w$ cannot fix another vertex $v'\,\ne\, v$.
Suppose otherwise. Since vertex stabilizers are closed,
the closed procyclic group $H\,=\,\overline{\langle w\rangle}$
fixes both $v$ and $v'$.
By \cite[Theorem~4.1.5]{RibesGraphs}, the nonempty
fixed-point set $T^H$ is a profinite subtree.
It therefore contains the minimal profinite subtree
$\Delta$ containing $v$ and $v'$, whose existence follows
from \cite[Proposition~2.4.9]{RibesGraphs}.
Thus $H$ fixes $\Delta$ pointwise.

The edge set of the standard profinite tree is closed,
because it is a finite disjoint union of compact coset
spaces. Hence $E(\Delta)\,=\,E(T)\cap\Delta$ is closed in
$\Delta$. By \cite[Proposition~2.1.6(c)]{RibesGraphs},
the subtree $\Delta$, which contains two distinct
vertices, has an edge adjacent to $v$.
This is the profinite-tree argument used in
\cite[proof of Lemma~6.6]{Xu25}; it does not assume
an ordinary finite edge path between arbitrary
profinite vertices.

Let $e$ be that edge. Its stabilizer is a conjugate, by an element
of $V$, of the completed group of an internal boundary torus of
the root piece. This torus is distinct from the exterior boundary
whose group is $P$. Since $w$ fixes $e$, it belongs to the
intersection of these two boundary groups in $V$.
Proposition~\ref{prop:local-boundary-intersections} makes the
intersection trivial at a hyperbolic root, or the fibre closure
at a Seifert root. Both conclusions contradict the preceding
observations about $w$. Thus $v$ is its unique fixed vertex.

Now let $a\,\in\, C_{\widehat{G_K}}(w)$. The equality $aw\,=\,wa$ implies
that $a$ takes vertices fixed by $w$ to vertices fixed by $w$.
By uniqueness $a$ fixes $v$, so $a\,\in\, V$.
In the case of a single JSJ piece we already have
$V\,=\,\widehat{G_K}$, so the same conclusion holds without the
tree argument.

Since $w\,\in\, P$ and $awa^{-1}\,=\,w$, it lies in
$P\cap aPa^{-1}$. The local intersection proposition forces
$a\,\in\, P$: otherwise the intersection is trivial in the
hyperbolic case, or only the fibre closure in the Seifert case.
Conversely $P$ is abelian and contains $w$, so every element
of $P$ commutes with $w$. This proves the centralizer equality.
\end{proof}

\begin{cor}\label{cor:word-peripheral}
For nontrivial knots, if a continuous isomorphism sends
$w_{K,n}$ to a conjugate $g w_{L,n}g^{-1}$, it sends
$\overline{P_K}$ to $g\overline{P_L}g^{-1}$.
After this one conjugation, the rooted JSJ correspondence and
the literal root restriction of
Lemma~\ref{lem:normalization} apply.
\end{cor}

\begin{proof}
An isomorphism preserves centralizers, and
$C(gwg^{-1})\,=\,gC(w)g^{-1}$. Apply
Lemma~\ref{lem:word-centralizer} in the two groups.
\end{proof}

\subsection{From the one element to an integral boundary matrix}

\begin{lem}\label{lem:word-integrality}
Let $K,\,L$ be nontrivial oriented tame knots and $n\,\ne\,0$.
Suppose a continuous isomorphism $\Phi :\widehat{G_K}\longrightarrow\widehat{G_L}$ satisfies
$\Phi(w_{K,n})\,=\,w_{L,n}$. Then $\Phi(P_K)=P_L$, and the restriction $\Phi|_{P_K}:P_K\longrightarrow P_L$
has a matrix in $\mathrm{GL}_2(\mathbb Z)$ with respect
to the ordered meridian--longitude bases
$(\mu_K,\lambda_K)$ and $(\mu_L,\lambda_L)$.
\end{lem}

\begin{proof}
Corollary~\ref{cor:word-peripheral} gives
$\Phi(\overline{P_K})\,=\,\overline{P_L}$.
The class of $w_{K,n}$ in the normalized completed
abelianization is $1$. Hence the induced automorphism of
$\Zh$ is multiplication by $1$:
\[
 \widehat\chi_L\Phi\ =\ \widehat\chi_K.
\]
If $A\,\in\,\GL_2(\Zh)$ is the boundary matrix, this equality
gives its first row. There exist $\rho\,\in\,\Zh$ and
$\beta\,\in\,\Zh^\times$ with
\begin{equation}\label{eq:word-boundary-shape}
 A=\begin{pmatrix}1&0\\ \rho&\beta\end{pmatrix},
 \qquad
 \Phi(\mu_K)=\mu_L\lambda_L^\rho,\quad
 \Phi(\lambda_K)=\lambda_L^\beta,\quad
 \rho+n\beta=n.
\end{equation}
We now use the root type, which is preserved by the rooted JSJ
correspondence.

\emph{Hyperbolic root.}
The literal root restriction of
Lemma~\ref{lem:normalization} is a continuous isomorphism of
hyperbolic piece completions with boundary map $A$.
The weaker peripheral theorem
Theorem~\ref{thm:hyp-regular} gives
$A\,=\,\eta B$ for $\eta\,\in\,\Zh^\times$ and
$B\,\in\,\GL_2(\Z)$.
Let $v\,=\,(1,\,n)^T$. Since $Av\,=\,v$, the first coordinate says
$1\,=\,\eta k$, where $k$ is the integer first coordinate of $Bv$.
Thus $k$ is an integer unit in $\Zh$, so $k\,=\,\pm 1$ and
$\eta\,=\,\pm 1$. It follows that $A\,\in\,\GL_2(\Z)$.
The stronger integral theorem
Theorem~\ref{thm:hyp-integral-regular} also gives this directly;
the displayed argument records that the one-mark step needs
only the older unit-regularity input.

\emph{Composing root.}
The target root is also composing. Indeed, such a root has at
least three boundary components, whereas a prime Seifert root
is a torus-knot exterior with one or a cable space with two.
The full peripheral correspondence preserves this number.
The characteristic fibre closures are generated by the two
meridians, so
$\Phi(\mu_K)\,=\,\mu_L^u$ for $u\,\in\,\Zh^\times$.
normalized abelianization gives $u\,=\,1$.
Then \eqref{eq:word-boundary-shape} has $\rho\,=\,0$ and
$n\beta\,=\,n$. Multiplication by the nonzero integer $n$
is injective on $\Zh$, so $\beta\,=\,1$ and $A\,=\,I$.

\emph{Prime Seifert root.}
Both root fibre elements have the coordinates of
Lemma~\ref{lem:fibre-slope}:
\[
 h_K\,=\,\mu_K^{c_K}\lambda_K,\quad
 h_L\,=\,\mu_L^{c_L}\lambda_L,\quad |c_K|,\, |c_L|\,\ge\,2.
\]
The root restriction transports the characteristic fibre closure,
so $\Phi(h_K)\,=\,h_L^u$ with $u\,\in\,\Zh^\times$.
Abelianization yields $c_K\,=\,c_Lu$.
The integral scalar lemma gives $u\,=\,\varepsilon\in\{1,\,-1\}$.
Thus, $A$ sends the two integer vectors
$(1,\,n)^T$ and $(c_K,\,1)^T$ to
$(1,\,n)^T$ and $\varepsilon(c_L,\,1)^T$ respectively.
They are independent over $\mathbb Q$, because
$1-nc_K\,\ne\,0$.
In matrix form,
\[
 A\begin{pmatrix}1&c_K\\n&1\end{pmatrix}
 \ =\
 \begin{pmatrix}1&c_L\\n&1\end{pmatrix}
 \diag(1,\,\varepsilon).
\]
Lemma~\ref{lem:integral-vectors} now makes all entries of
$A$ integers. Since $A$ is already invertible over $\Zh$,
it follows from \eqref{eq:integral-intersection} that $A\,\in\,\GL_2(\Z)$.
This finishes all cases.
\end{proof}

\subsection{Why an integral map preserves the two slopes}
The next statement is the meridian--longitude consequence of
\cite[Lemma~10.4]{XuRegular}. We include the proof and identify
the classical inputs, since it is the point at which knot
topology rules out the torus automorphism.

\begin{lem}\label{lem:integral-axes}
Let $K,\,L$ be nontrivial knots. Suppose a continuous isomorphism
$f\,:\,\widehat{G_K}\,\longrightarrow\,\widehat{G_L}$ carries the exterior
peripheral closures to one another and has an integral normalized
boundary matrix. After carrying them to the chosen representatives,
\[
 f(\mu_K)\ =\ \mu_L^{\epsilon_\mu},\quad
 f(\lambda_K)\ =\ \lambda_L^{\epsilon_\lambda},
 \quad \epsilon_\mu,\epsilon_\lambda\ \in\ \{1,\,-1\}.
\]
\end{lem}

\begin{proof}
normalize by the common carrying conjugation and write
$A\,=\,\left(\begin{smallmatrix}a&b\\c&d\end{smallmatrix}\right)
\,\in\,\GL_2(\Z)$.
An isomorphism of completed abelianizations is multiplication
by a unit $\alpha$. Compatibility with linking number makes
the first row of $A$ equal to $(\alpha,\,0)$.
Thus, $b\,=\,0$ and $a\,=\,\alpha$ is both an integer and a profinite
unit, so $a\,=\,\pm1$. The determinant condition gives $d\,=\,\pm 1$.
Consequently $f(\lambda_K)\,=\,\lambda_L^d$, and the longitude
axis is already preserved.

The other column is the primitive integer slope
$s\,=\,\mu_L^a\lambda_L^c$.
The source meridian closed normally generates
$\widehat{G_K}$, so its image $s$ closed normally generates
$\widehat{G_L}$. Let $E_L(s)$ denote the closed manifold obtained
by Dehn filling $E_L$ along the unoriented slope of $s$.
The Van Kampen theorem gives
\[
 \pi_1(E_L(s))\ =\
 G_L/\langle\!\langle s\rangle\!\rangle_{G_L}.
\]
Lemma~\ref{lem:quotient-completion} therefore gives
\[
 \widehat{\pi_1(E_L(s))}\ \cong\  \widehat{G_L}/
 \overline{\langle\!\langle s\rangle\!\rangle}_{\widehat{G_L}}\ =\ 1.
\]
Fundamental groups of compact $3$-manifolds are residually finite, using
Hempel's residual-finiteness results and geometrization;
see \cite{HempelRF,PerelmanSurgery} and the modern account
\cite{AFW}. Hence $\pi_1(E_L(s))\,=\,1$.
The Poincar\'e conjecture, proved by Perelman
\cite{PerelmanEntropy,PerelmanSurgery,PerelmanExtinction}, implies
$E_L(s)\,\cong\, S^3$.

Gordon--Luecke's filling theorem
\cite[Theorem~2]{GordonLuecke} says that a nontrivial knot
exterior has no nonmeridional Dehn filling homeomorphic to $S^3$.
Thus $s$ is the meridian slope, so $c\,=\,0$.
Its orientation is one of the two possible signs, giving
$f(\mu_K)\,=\,\mu_L^{\pm 1}$.
\end{proof}

\subsection{Proof of Theorem \ref{thm:one-recovery} of introduction}

\begin{proof}[Proof of Theorem~\ref{thm:one-recovery}]
Compose the given $\Phi$ with $\Inn(g^{-1})$ once.
The new map sends $w_K$ exactly to $w_L$.

If one knot is the unknot, its completion is the abelian group
$\Zh$. A nontrivial knot cannot have abelian completion:
residual finiteness would make its discrete group abelian,
then its cyclic abelianization would make the group $\Z$,
contradicting the fact that its embedded boundary group is $\Z^2$.
Thus, an isomorphism involving the unknot has the unknot on
both sides. Its longitude is $1$ and $w\,=\,\mu$, so the assertion
is immediate.

For nontrivial knots, peripheral recovery is
Corollary~\ref{cor:word-peripheral} with $n\,=\,1$.
Lemma~\ref{lem:word-integrality} makes the normalized boundary
map integral. Lemma~\ref{lem:integral-axes} then gives
\[
 \Phi(\mu_K)\ =\ \mu_L^{\epsilon_\mu},\qquad
 \Phi(\lambda_K)\ =\ \lambda_L^{\epsilon_\lambda}.
\]
The exact word equation is
$\mu_L^{\epsilon_\mu}\lambda_L^{\epsilon_\lambda}
\, =\, \mu_L\lambda_L$.
In the independent coordinates of
$\overline{P_L}\,\cong\,\Zh^2$, it says that
$\epsilon_\mu\,=\,1$ and $\epsilon_\lambda\,=\,1$.
Undoing the one initial inner automorphism gives both original
equalities with the original element $g$.
No new separate conjugators have been introduced.
\end{proof}

\begin{cor}\label{cor:other-word}
For any fixed integer $n\,\ne\,0$, a continuous knot-group
completion isomorphism taking $\mu_K\lambda_K^n$ to
$g\mu_L\lambda_L^ng^{-1}$ preserves $\mu_K,\,\lambda_K$
with the same $g$. Consequently, marking $\mu_K\lambda_K^n$
also recognizes oriented tame knots.
\end{cor}

\begin{proof}
The preceding centralizer and integrality lemmas were proved
for this $n$. After applying the axes lemma, the word equation
gives $\epsilon_\mu\,=\,1$ and
$n\epsilon_\lambda\,=\,n$. Injectivity of multiplication by
$n$ on $\Zh$ gives $\epsilon_\lambda\,=\,1$.
The unknot case is unchanged. Apply ordered recognition.
\end{proof}

\subsection{Proof of Theorem~\ref{thm:main} of introduction}

\begin{proof}[Proof of Theorem~\ref{thm:main}]
Corollary~\ref{cor:ordered-recognition} gives $(i) \Leftrightarrow (iii)$. Since an ordered marked isomorphism preserves the product and thus gives a one-element marked isomorphism, hence $(iii) \Rightarrow (ii)$. For $(ii) \Rightarrow (iii)$. Theorem~\ref{thm:one-recovery} shows that the single marked element recovers the exact ordered pair, giving $(ii)\Rightarrow(iii)$. Thus $(i)$, $(ii)$, and $(iii)$ are equivalent.
\end{proof}

\section{An exact meridian and a closed longitude subgroup}
\label{sec:reduced}

\begin{prop}\label{prop:unit-marking-regularity}
Let $K,\,L$ be nontrivial oriented tame knots. Suppose a continuous
isomorphism $\Phi\,:\,\widehat{G_K}\,\longrightarrow\,\widehat{G_L}$ satisfies, up
to simultaneous conjugacy, the following:
\[
 \Phi(\mu_K)\ =\ \mu_L^\alpha,\quad
 \Phi(\lambda_K)\ =\ \lambda_L^\beta,\quad
 \alpha,\beta\ \in\ \Zh^\times.
\]
Then $\beta\,=\,\epsilon\alpha$ for some $\epsilon\,\in\,\{1,\,-1\}$.
In particular its exterior matrix is unit-regular:
\[
 \diag(\alpha,\,\beta)\ =\ \alpha\diag(1,\,\epsilon).
\]
\end{prop}

\begin{proof}
Normalize the common conjugation. First suppose that $K$ is prime.
Then $L$ is prime, because otherwise Theorem~\ref{thm:constituents}
applied to $\Phi^{-1}$, with exponents
$\alpha^{-1},\,\beta^{-1}$, would make $K$ composite.
Proposition~\ref{prop:jsj} and
Lemma~\ref{lem:normalization} match the root groups while
retaining the specified exterior matrix.

If the root is hyperbolic, note that
Theorem~\ref{thm:hyp-regular} gives the following:
\[
 \diag(\alpha,\,\beta)\ =\ \eta B,\quad
 \eta\ \in\ \Zh^\times,\quad B\ \in\ \GL_2(\Z).
\]
The off-diagonal entries of $B$ vanish, because $\eta$ is a unit.
A diagonal matrix in $\GL_2(\Z)$ has both diagonal entries in
$\{1,\,-1\}$. Thus
$\alpha=\eta\epsilon_1$ and $\beta=\eta\epsilon_2$, with
$\epsilon_1,\,\epsilon_2\,\in\,\{1,\,-1\}$, and $\beta/\alpha$ is a sign.

If the root is Seifert, then use Lemma~\ref{lem:fibre-slope}.
The characteristic fibre closures give
$\Phi(h_K)\,=\,h_L^u$ for a unit $u$, with
$h_K\,=\,\mu_K^{c_K}\lambda_K$ and
$h_L\,=\,\mu_L^{c_L}\lambda_L$.
Comparison in the embedded exterior boundary group yields
\[
 c_K\alpha\ =\ c_Lu,\qquad \beta\ =\ u,
 \]
 hence $c_K\,=\,c_L\beta\alpha^{-1}$.
The integers $c_K,\,c_L$ are nonzero. The integral scalar lemma
therefore gives $\beta\alpha^{-1}\,\in\,\{1,\,-1\}$.

If $K$ is composite, Theorem~\ref{thm:constituents} produces
matching prime-factor isomorphisms with the same $\alpha$ and
$\beta$. Apply the prime argument to any factor. The asserted
unit-regular matrix follows immediately.
\end{proof}

\subsection{Proof of Theorem~\ref{thm:reduced} of introduction}

\begin{proof}[Proof of the first assertion of Theorem~\ref{thm:reduced}]
Apply Proposition~\ref{prop:unit-marking-regularity} with
$\alpha\,=\,1$. It gives $\beta\,=\,\epsilon\alpha\,=\,\epsilon$ with
$\epsilon\in\{1,\,-1\}$. The argument uses only peripheral
regularity, characteristic Seifert fibres, and prime coherence.
\end{proof}

\begin{proof}[Proof of the second assertion of Theorem~\ref{thm:reduced}]
Let $\Rk K\,\cong\,\Rk L$. If one knot is the unknot,
its longitude subgroup is trivial, whereas a nontrivial knot
has a nontrivial embedded longitude closure. Both knots are
therefore unknots in that case.

For nontrivial knots the isomorphism sends the longitude
generator to a unit power $\lambda_L^\beta$, using the same
carrying conjugation that identifies the meridians. The first
assertion gives $\beta\,=\,1$ or $\beta\,=\,-1$. If $\beta\,=\,1$,
Theorem~\ref{thm:main} gives $K\,\simeq\, L$. If $\beta\,=\,-1$,
compose with the natural reversed-mirror map for $L$ described
in \eqref{eq:dagger-marks}. It fixes the positive meridian and
inverts the preferred longitude. The composite is consequently
an exact ordered marked isomorphism to $L^\dagger$, and
Theorem~\ref{thm:main} gives $K\,\simeq\, L^\dagger$.

Conversely an oriented equivalence preserves the reduced marking,
and the reversed-mirror map preserves the meridian and the
longitude subgroup. This proves both directions.

\end{proof}

The reversed-mirror ambiguity is necessary for the trefoil. For a
composite knot, Theorem~\ref{thm:constituents} forces the same sign on
every prime branch. Thus, the ambiguity is global and at most twofold;
it disappears for knots equivalent to their reversed mirrors.

\section{Finite representations and colouring consequences}
\label{sec:applications}

\begin{proof}[Proof of Corollary~\ref{cor:colouring-main}]
Theorem~\ref{thm:main} identifies oriented knot
equivalence with both $\Wk K\,\cong\,\Wk L$ and $\Pk K\,\cong\,\Pk L$.
Proposition~\ref{prop:finite} with one mark identifies the first
isomorphism with equality of $\Qwk K,\,\Qwk L$ and of all counts
$N_Q(z;K),\,N_Q(z;L)$. With two marks it identifies the second
isomorphism with equality of $\Qk K,\,\Qk L$ and of all
meridian--longitude counts. These are exactly the coefficients
of the colouring polynomials by
\eqref{eq:polynomial-coefficient}. This proves all five
equivalences.
\end{proof}

Equation~\eqref{eq:one-count-from-polynomial} expresses the
one-word counts as sums of polynomial coefficients. The marked
profinite dictionary proves their completeness across all finite groups without inverting that summation in one target. Abelian groups alone cannot see a nontrivial longitude image, since the longitude belongs to the commutator subgroup.

\begin{cor}\label{cor:finite-witness}
If oriented knots $K,\,L$ are inequivalent, a finite
quotient with one marked element occurs for one and not the other,
and there are a finite group $F$ and $z\,\in \,F$ such that
$N_F(z;K)\,\ne\, N_F(z;L)$. There is also a finite ordered
peripheral quotient occurring for just one knot, and a finite
group $Q$ and $x\,\in\, Q$ with $P_Q^x(K)\,\ne\, P_Q^x(L)$.
\end{cor}
\begin{proof}
Apply Theorem~\ref{thm:main} and
Corollary~\ref{cor:finite-witness-algebraic}.
\end{proof}

From a specified finite group, the differing  counts can be verified by enumerating generator assignments, checking the relators and evaluating the marked word. Thus enumeration of finite groups will eventually separate inequivalent knots. The argument gives no bound on the order of a distinguishing group. Agreement on finitely many finite groups does not provide a stopping criterion for equivalence. 

\subsection{Connected-sum multiplicativity}
The meridional amalgam in \eqref{eq:meridional-amalgam} explains
Eisermann's product formula
\begin{equation}\label{eq:polynomial-product}
 P_Q^x(K_1\#K_2)\ =\ P_Q^x(K_1)P_Q^x(K_2).
\end{equation}
A representation of the amalgam with meridian image $x$ is
exactly a pair of representations of the factors with that
meridian image. The longitude image is the ordered product of
their longitude images. Expanding the finite sums in $\Z[Q]$
gives the formula \cite[Proposition~2.5 and Corollary~2.6]{Eisermann}.
This identity does not give a unique factorization theorem in
the group ring, which may have zero divisors. Prime-summand
recovery instead comes from the composing-space geometry and
the common-exponent argument of Theorem~\ref{thm:constituents}.

\section{The square and granny knots in finite groups}
\label{sec:finite-example}

For this section, a finite peripheral pair is
$(Q,\,q(P_K))$ for an epimorphism $q\,:\,G_K\,\onto\, Q$.
An isomorphism of such pairs is a group isomorphism carrying one
displayed subgroup onto the other, without selected generators.

We use one trefoil calculation to obtain two finite witnesses:
a peripheral-subgroup image in $A_5\times C_5$, and a
prescribed-image count for $\mu\lambda$ in $A_5$. The full
verification program is included in Appendix~\ref{sec:script}.
Both calculations are elementary and independent of the
geometric realization theorems.

\subsection{A cyclic extension converts longitude order into subgroup order}

A group $F$ is \emph{perfect} if $[F,\,F]\,=\,F$; equivalently its
abelianization is trivial. Let $C_n\,=\,\langle z\,\,\big\vert\,\, z^n\,=\,1\rangle$
denote the cyclic group of order $n$.

\begin{prop}\label{prop:cyclic-lift}
Let $\rho\,:\,G_K\,\onto\, F$ be an epimorphism to a finite perfect group.
Put $a\,=\,\rho(\mu_K)$ and $b\,=\,\rho(\lambda_K)$, and let $n$ be a
positive multiple of the order of $a$. Define
\[
 q\ :\ G_K\ \longrightarrow\ F\times C_n,\quad
 g\ \longmapsto\ \bigl(\rho(g),\,z^{\chi_K(g)}\bigr).
\]
Then $q$ is surjective, and
\[
 q(P_K)\ =\ \langle(a,z),(b,1)\rangle
\ \cong\ C_n\times C_{\operatorname{ord}(b)}.
\]
In particular, $|q(P_K)|\,=\,n\,\operatorname{ord}(b)$.
\end{prop}

\begin{proof}
Both coordinate maps are homomorphisms. Write $H\,=\,q(G_K)$.
Its projection to $F$ is surjective, so its commutator subgroup
projects onto $[F,\,F]\,=\,F$. But every commutator in $F\times C_n$
lies in $F\times\{1\}$. Thus $[H,\,H]\,=\,F\times\{1\}$.
The second projection of $H$ is also surjective because
$\chi_K(\mu_K)\,=\,1$, so $H\,=\,F\times C_n$.

The longitude has $\chi_K$-value zero, giving the displayed
peripheral generators. They commute because the original two
marks commute. The element $(a,\,z)$ has order $n$, and its cyclic
group has trivial intersection with
$\langle(b,\,1)\rangle$: an element of the former with second
coordinate $1$ has exponent divisible by $n$ and is the identity.
The two commuting cyclic groups therefore give their direct
product.
\end{proof}

This construction converts a difference in longitude orders into a difference in the orders of peripheral subgroup images.
To distinguish the knots using the subgroup image with no chosen generators, one must
consider all possible meridian images. We do so for the square and granny knots below.

\subsection{The six trefoil maps with one fixed meridian}

The $A_5$ trefoil polynomial and its six colourings are already
Eisermann's \cite[Example~1.3]{Eisermann}. We give the presentation calculation and complete list to derive the connected-sum counts and the counts with prescribed image of $\mu\lambda$.

Let $T$ be a right-handed trefoil and $T^\dagger$ its reversed
mirror. Put $K_g\,=\,T\#T$ and $K_s\,=\,T\#T^\dagger$, with the indicated
orientations; these are the granny and square knots. We use
\[
 G_T\ =\ \langle x,\,y\,\,\big\vert\,\, x^2\,=\,y^3\rangle,\quad
 \mu\ =\ xy^{-1},\quad \lambda\ =\ x^2\mu^{-6}.
\]
The normalization agrees with the positive trefoil fibre
$\mu^6\lambda\,=\,x^2$. For the reversed mirror we identify its
abstract group with this presentation while replacing the
longitude by $\lambda^{-1}$ and retaining the meridian.

The alternating group $A_5$ is the group of even permutations of
five symbols; it has order $60$, is simple and is perfect.
Cycle notation $(i_1\cdots i_k)$ sends each listed symbol to the
next and the last to the first; products of permutations act
from right to left. Fix $a\,=\,(1\,2\,3\,4\,5)$.
For a map with $\mu\,\longmapsto\, a$, its values satisfy
$x\,=\,ay$ and $(ay)^2\,=\,y^3$. There are exactly six choices, as follows.
\begin{center}
\begin{tabular}{@{}lll@{}}
\toprule
$x$ & $y$ & Image group\\
\midrule
$a^3$ & $a^2$ & $\langle a\rangle$\\
$(1\,2)(4\,5)$ & $(2\,5\,3)$ & $A_5$\\
$(1\,2)(3\,4)$ & $(2\,5\,4)$ & $A_5$\\
$(1\,5)(3\,4)$ & $(1\,4\,2)$ & $A_5$\\
$(1\,5)(2\,3)$ & $(1\,4\,3)$ & $A_5$\\
$(2\,3)(4\,5)$ & $(1\,5\,3)$ & $A_5$\\
\bottomrule
\end{tabular}
\end{center}
Here is a direct way to check completeness of the list. The
element $h\,=\,x^2\,=\,y^3$ centralizes $x,\,y$, hence $a$. The centralizer
of a $5$-cycle in $A_5$ is $\langle a\rangle$, so $h\,=\,a^e$.
If $e\,\ne\,0$ modulo $5$, both $x,\,y$ centralize a generator of
this cyclic group and belong to it. Solving $x\,=\,ay$ and
$x^2\,=\,y^3$ gives the single cyclic solution $x\,=\,a^3,\,y\,=\,a^2$.
If $e\,=\,0$, then $x^2\,=\,y^3\,=\,1$. Since $xy^{-1}$ has order $5$,
$x$ has order $2$ and $y$ order $3$. Checking the $20$ possible
$3$-cycles $y$ gives precisely the five noncyclic rows above.
They generate $A_5$: the generated group contains elements of
orders $2,3,5$, so its order is $30$ or $60$, and an index-two
subgroup of the simple group $A_5$ is impossible. The program
in Appendix~\ref{sec:script} checks all $60$ values of $y$ and
verifies the generated subgroup orders independently.

In the cyclic row $\lambda$ maps to $1$. In each of the other
five rows it maps to $a^{-6}=a^{-1}$. For $T^\dagger$ the
five corresponding longitude images are $a$.

\subsection{Connected-sum counts and peripheral images}
A connected-sum map with fixed meridian image $a$ is a choice
of one of the six maps on each factor. Exactly one of the $36$
choices has both factors cyclic. The remaining $35$ are
surjective to $A_5$, since at least one factor already surjects.
The longitude images are products, giving this table:
\begin{center}
\begin{tabular}{@{}>{\raggedright\arraybackslash}p{0.37\textwidth}cll@{}}
\toprule
Factor images & Number & Granny longitude & Square longitude\\
\midrule
First full, second cyclic & $5$ & $a^{-1}$ & $a^{-1}$\\
First cyclic, second full & $5$ & $a^{-1}$ & $a$\\
Both full & $25$ & $a^{-2}$ & $1$\\
\bottomrule
\end{tabular}
\end{center}
The table follows by multiplying the longitude images from the six trefoil maps.

Lift these epimorphisms to $Q\,=\,A_5\times C_5$ by
Proposition~\ref{prop:cyclic-lift}. All have meridian image
$(a,\,z)$. For the granny knot every such peripheral image has
order $25$. For the square knot, $25$ of the maps have longitude
image $1$ and peripheral image of order $5$.

\begin{thm}\label{thm:granny-square}
There is a quotient $q\,:\,G_{K_s}\,\onto\, A_5\times C_5$ whose
peripheral image has order $5$. There is no quotient
$G_{K_g}\onto A_5\times C_5$ with peripheral image of order $5$.
Thus the sets of finite quotient groups equipped with the image
of the peripheral subgroup distinguish the square and granny knots.
\end{thm}

\begin{proof}
The square witness is any lifted map for which both factors are
full. Suppose a granny epimorphism to $A_5\times C_5$ had
peripheral image of order $5$. Its $A_5$ projection is
surjective, so the meridian image there is nonidentity:
a meridian normally generates the knot group. Its order
divides $5$, hence it is a $5$-cycle. The $C_5$ projection
sends the meridian to a generator, since the map is surjective
and the knot abelianization is generated by that meridian.
Conjugation by a permutation of the five symbols carries any
$5$-cycle to $a$ and is an automorphism of $A_5$; an automorphism
of $C_5$ carries that generator to $z$.

After applying these target automorphisms we are in the normalized
situation of the table. The cyclic component is exactly
$\chi_{K_g}$ modulo $5$, and every full $A_5$ map has longitude
$a^{-1}$ or $a^{-2}$. Its peripheral image therefore has order
$25$, a contradiction. Subgroup order is preserved by every
isomorphism of the finite target, so the distinction does not
require any named generators.
\end{proof}

\subsection{A one-element count in \texorpdfstring{$A_5$}{A5}}

The same six-map calculation gives a direct example for the
new one-element marking.

\begin{prop}\label{prop:one-word-A5}
For the oriented granny and square knots fixed above,
\[
 N_{A_5}(1;K_g)\ =\ 241,\qquad N_{A_5}(1;K_s)\ =\ 121.
\]
Among epimorphisms to $A_5$, the respective counts with
$\rho(\mu\lambda)=1$ are $240$ and $120$.
\end{prop}
\begin{proof}
Fix a meridian image $a$ of order $5$.
Multiplying each longitude in the preceding table by $a$
shows that $w$ has image $1$ for $10$ granny maps and $5$
square maps: in the square case only the first-full,
second-cyclic row contributes. These maps are all surjective.
The one additional pair with both factors cyclic has longitude
$1$ and hence word image $a$, not $1$.
There are $24$ elements of order $5$ in $A_5$.
Conjugation by a permutation, possibly odd, induces an
automorphism of $A_5$ carrying any chosen $5$-cycle to $a$,
so the count is the same for all of them. This contributes
$24\cdot10=240$ and $24\cdot5=120$.

We check that the other meridian orders add only the trivial
representation. A meridian image $1$ makes the whole map
trivial by normal generation. Suppose its order is $2$ or $3$
and consider either trefoil factor. Put $h=x^2=y^3$ as above;
it centralizes the meridian image. For an element of order $2$,
its centralizer in $A_5$ is a Klein four group, namely
$C_2\times C_2$, whose nonidentity elements all have order $2$.
But no square in $A_5$ has order $2$, since $A_5$ has no
element of order $4$. Hence $h=1$.
For a meridian of order $3$, its centralizer is its cyclic
group of order $3$. If $h\ne1$, both $x,y$ centralize its
generator and the image is cyclic; otherwise $h=1$ already.
In the cyclic case the preferred longitude is trivial because
it belongs to the commutator subgroup. In the other cases
$\lambda=h\,\mu^{-6}=1$.
Thus both summand longitudes are trivial for these meridian
orders and $w$ has the nonidentity meridian image.
The only remaining contribution to $N_{A_5}(1;-)$ is therefore
the one trivial representation, giving $241$ and $121$.
\end{proof}

The complete program in Appendix~\ref{sec:script} enumerates all
meridian images as an independent check. Both knots have $2220$
homomorphisms to $A_5$ in total, but the distributions of the
image of $w$ differ. Thus forgetting all prescribed word images
loses the distinction in this target.

\appendix
\section{The cusped input and finite verification}
\subsection{Outline of Xu's cusped genuineness argument}\label{subsec:xu-cusped-proof}
The proof of \cite[Proposition~16.3]{Xu26} first matches integral cusp
curves $a_i,b_i$ by peripheral regularity. For large $n$, the quotients
\[
 \pi_1(M)/\langle\!\langle a_1^{n!},\ldots,a_k^{n!}\rangle\!\rangle,
 \qquad
 \pi_1(N)/\langle\!\langle b_1^{n!},\ldots,b_k^{n!}\rangle\!\rangle
\]
are closed hyperbolic orbifold filling groups, and the given profinite
map descends to their completions. These fillings impose powers of
slopes, giving finite cyclic local groups on the attached cores.
Xu's uniform-lattice theorem \cite[Theorem~15.3]{Xu26} makes the
descended maps genuine. Mostow rigidity gives orbifold isometries;
drilling their singular cores gives homeomorphisms $M\to N$.

Finiteness of their isotopy classes gives an infinite subsequence
inducing one discrete isomorphism $\alpha$, up to conjugacy. Every
finite quotient factors through all sufficiently late factorial
fillings, since its element orders divide $n!$. Thus the given map and
$\widehat\alpha$ agree up to conjugacy on cofinal finite characteristic
quotients. Compactness supplies one common conjugator
\cite[Lemma~6.3 and Corollary~6.4]{Xu26}. This recovers the particular
profinite isomorphism. Orbifold power relations, unlike ordinary
manifold filling relations, eventually vanish in every finite quotient
\cite[Remark~16.6]{Xu26}.

\subsection{Verification program}\label{sec:script}

The following standard-library Python program enumerates all trefoil
maps and connected-sum representations in $A_5$. Permutations are tuples
on $\{0,1,2,3,4\}$, and multiplication composes right to left. The
subgroup routine exhausts the finite subgroup generated by its inputs.
The assertions check the six fixed-meridian maps, both longitude
distributions, the peripheral orders from Proposition~\ref{prop:cyclic-lift},
and all one-word counts.
\begin{lstlisting}
from itertools import permutations
from collections import Counter

identity = tuple(range(5))

def multiply(p, q):
    return tuple(p[q[i]] for i in range(5))

\end{lstlisting}
\Needspace{17\baselineskip}
\begin{lstlisting}
def power(p, n):
    if n < 0:
        inverse = [0] * 5
        for i, j in enumerate(p):
            inverse[j] = i
        return power(tuple(inverse), -n)
    result = identity
    while n:
        if n & 1:
            result = multiply(result, p)
        p = multiply(p, p)
        n //= 2
    return result

def even(p):
    return sum(p[i] > p[j] for i in range(5)
               for j in range(i + 1, 5)) % 2 == 0

def generated(generators):
    subgroup, todo = {identity}, [identity]
    while todo:
        g = todo.pop()
        for h in generators:
            gh = multiply(g, h)
            if gh not in subgroup:
                subgroup.add(gh)
                todo.append(gh)
    return subgroup

def order(p):
    n, value = 1, p
    while value != identity:
        value = multiply(value, p)
        n += 1
    return n

A5 = [p for p in permutations(range(5)) if even(p)]
a = (1, 2, 3, 4, 0)  # (1 2 3 4 5)

def trefoil_maps(m):
    maps = []
    for y in A5:
        x = multiply(m, y)  # mu = x y^(-1)
        if power(x, 2) == power(y, 3):
            b = multiply(power(x, 2), power(m, -6))
            maps.append((x, y, b))
    return maps

fixed = trefoil_maps(a)
assert len(fixed) == 6
assert Counter(len(generated((x, y))) for x, y, b in fixed) == {5: 1, 60: 5}
assert Counter(b for x, y, b in fixed) == {identity: 1, power(a, -1): 5}

\end{lstlisting}
\Needspace{8\baselineskip}
\begin{lstlisting}
word = [Counter(), Counter()]
epi_word = [Counter(), Counter()]
fixed_longitude = [Counter(), Counter()]
for m in A5:
    maps = trefoil_maps(m)
    for x1, y1, b1 in maps:
        for x2, y2, b2 in maps:
            onto = len(generated((x1, y1, x2, y2))) == 60
            longitudes = (multiply(b1, b2), multiply(b1, power(b2, -1)))
            for i, b in enumerate(longitudes):
                w = multiply(m, b)
                word[i][w] += 1
                if onto:
                    epi_word[i][w] += 1
                if m == a:
                    fixed_longitude[i][b] += 1

g, s = fixed_longitude
assert g == Counter({power(a, -1): 10, power(a, -2): 25})
assert s == Counter({power(a, -1): 5, a: 5, identity: 25})
assert sum(g.values()) == sum(s.values()) == 35
assert {5 * order(b) for b in g} == {25}
assert {5 * order(b) for b in s} == {5, 25}
assert [c[identity] for c in word] == [241, 121]
assert [c[identity] for c in epi_word] == [240, 120]
assert [sum(c.values()) for c in word] == [2220, 2220]
print("Identity word counts:", [c[identity] for c in word])
print("Surjective identity counts:", [c[identity] for c in epi_word])
print("Total representations:", [sum(c.values()) for c in word])
\end{lstlisting}
The output is:
\begin{verbatim}
Identity word counts: [241, 121]
Surjective identity counts: [240, 120]
Total representations: [2220, 2220]
\end{verbatim}

\end{document}